\documentclass[a4paper,10pt,leqno,english]{smfart}
\usepackage{aeguill}
\usepackage{enumerate}
\usepackage{amssymb,amsmath,latexsym,amsthm}
\usepackage[T1]{fontenc}
\usepackage{smfthm}      
\usepackage{geometry}   
\usepackage{url}      
\usepackage[frenchb, english]{babel}
\usepackage[utf8]{inputenc}   
\usepackage{mathrsfs}
\usepackage{xcolor}
\usepackage{comment}
\usepackage{dsfont}
\usepackage{yhmath}
\definecolor{violet}{rgb}{0.0,0.2,0.7}
\definecolor{rouge}{cmyk}{0.0,0.6,0.4,0.3}
\definecolor{rouge2}{rgb}{0.8,0.0,0.2}
\usepackage{hyperref}
	
\hypersetup{
    bookmarks=true,         
    unicode=false,          
    pdftoolbar=true,        
    pdfmenubar=true,        
    pdffitwindow=false,     
    pdfstartview={FitH},    
    pdftitle={},    
    pdfauthor={},     
    colorlinks=true,       
   linkcolor=violet,          
    citecolor=rouge2,        
    filecolor=black,      
    urlcolor=cyan}           
\newcommand{\R}{\mathbb{R}}
\newcommand{\CC}{\mathbb{C}}
\newcommand{\Q}{\mathbb{Q}}
\newcommand{\Z}{\mathbb{Z}}

\renewcommand{\b}{\bar}
\renewcommand{\d}{\partial}

\newcommand{\vp}{\varphi}

\renewcommand{\O}{\mathcal{O}}
\newcommand{\Ox}{\mathcal{O}_{X}}

\newcommand{\ep}{\varepsilon}

\newcommand{\la}{\langle}
\newcommand{\ra}{\rangle}

\renewcommand{\ge}{\geqslant}
\renewcommand{\le}{\leqslant}

\newcommand{\Ric}{\mathrm{Ric} \,}
\newcommand{\om}{\omega}

\newcommand{\ddc}{dd^c}
\newcommand{\xmd}{X_0}

\newcommand{\vpe}{\varphi_{\varepsilon}}

\newcommand{\fe}{F_{\varepsilon}}

\newcommand{\D}{\Delta}
\newcommand{\De}{\Delta_{\omega_{\varepsilon}}}
\renewcommand{\fe}{f_{\varepsilon}}
\newcommand{\Supp}{\mathrm {Supp}}
\newcommand{\Dlc}{\Delta_{lc}}
\newcommand{\Dklt}{\Delta_{klt}}
\newcommand{\Xlc}{X_{lc}}

\newcommand{\cka}{\mathscr{C}^{k,\alpha}_{qc}}

\renewcommand{\b}{\bar}
\renewcommand{\om}{\omega}
\newcommand{\omi}{\omega_{\infty}}
\newcommand{\omke}{\omega_{\textrm{KE}}}

\newcommand{\ome}{\omega_{\varepsilon}}

\newcommand{\tr}{\mathrm{tr}}

\newcommand{\psie}{\psi_{\ep}}
\newcommand{\Fe}{F_{\ep}}

\newtheorem*{thma}{Theorem A}
\newtheorem*{thmb}{Theorem B}

\begin{document}
\title[Kähler-Einstein metrics with mixed Poincaré and cone singularities]{Kähler-Einstein metrics with mixed Poincaré and cone singularities along a normal crossing divisor}

\date{\today}
\author{Henri Guenancia}
\address{Institut de Mathématiques de Jussieu \\
Université Pierre et Marie Curie \\
 Paris 
\& Département de Mathématiques et Applications \\
\'Ecole Normale Supérieure \\
Paris}
\email{guenancia@math.jussieu.fr}
\urladdr{www.math.jussieu.fr/~guenancia}

\begin{abstract} 
Let $X$ be a Kähler manifold and $\D$ be a $\R$-divisor with simple normal crossing support and coefficients between $1/2$ and $1$. Assuming that $K_X+\D$ is ample, we prove the existence and uniqueness of a negatively curved Kahler-Einstein metric on $X\setminus \Supp(\D)$ having mixed Poincaré and cone singularities according to the coefficients of $\D$. As an application we prove a vanishing theorem for certain holomorphic tensor fields attached to the pair $(X,\D)$.
\end{abstract}

\maketitle

\tableofcontents

\newpage

\section*{Introduction}

Let $X$ be a compact Kähler manifold of dimension $n$, and $\D=\sum a_i \D_i$ an effective $\R$-divisor with simple normal crossing support such that the $a_i$'s satisfy the following inequality: $0 < a_i \le 1$. We write $X_0=X\setminus \Supp(\D)$.\\

Our local model is given by the product $X_{\rm mod}=(\mathbb{D}^*)^r\times (\mathbb{D}^*)^s \times \mathbb{D}^{n-(s+r)}$ where $\mathbb{D}$ (resp. $\mathbb{D}^*$) is the disc (resp. punctured disc) of radius $1/4$ in $\mathbb C$, the divisor being $D_{\rm mod}=d_1 [z_1=0]+\cdots+d_r [z_r=0]+[z_{r+1}=0]+\cdots + [z_{r+s}=0]$, with $d_i<1$. We will say that a metric $\om$ on $X_{\rm mod}$ has mixed Poincaré and cone growth (or singularities) along the divisor $D_{\rm mod}$ if there exists $C>0$ such that 
\[C^{-1} \omega_{\rm mod} \le \om \le C \,\om_{\rm mod}\]
where
\[\om_{\rm mod}:=\sum_{j=1}^r \frac{i dz_j\wedge d\bar z_j}{|z_j|^{2d_j}} + \sum_{j=r+1}^s \frac{i dz_j\wedge d\bar z_j}{|z_j|^{2} \log^{2}|z_j|^{2}}+\sum_{j=r+s+1}^n i dz_j\wedge d\bar z_j \] 
is simply the product metric of the standard cone metric on $(\mathbb{D}^*)^r$, the Poincaré metric on $(\mathbb{D}^*)^s$, and the euclidian metric on $\mathbb{D}^{n-(s+r)}$.\\

This notion makes sense for global (Kähler) metrics $\om$ on the manifold $X_0$; indeed, we can require that on each trivializing chart of $X$ where the pair $(X,\D)$ becomes $(X_{\rm mod},D_{\rm mod})$ (those charts cover $X$), $\om$ is equivalent to $\om_{\rm mod}$ just like above, and this does not depend on the chosen chart.\\

Our goal will then be to find, whenever this is possible, Kähler metrics on $X_0$ having constant Ricci curvature and mixed Poincaré and cone growth along $\D$. Those metrics will naturally be called Kähler-Einstein metrics. For reasons which will appear in section \ref{sec:kep} and more precisely in Remark \ref{rem:bm}, we will restrict ourselves to looking for Kähler-Einstein metrics with negative curvature.\\

The existence of Kähler-Einstein metrics (in the previously specified sense) has already been studied in various contexts and for multiple motivations. The logarithmic case (all coefficients of $\D$ are equal to $1$) has been solved when $K_X+\D$ is assumed to be ample by R. Kobayashi \cite{KobR} and G.Tian-S.T.Yau \cite{Tia}, the latter considering also orbifold coefficients for the fractional part $\Dklt=\sum_{\{a_i<1\}} a_i \D_i $ of $\D$, that is of the form $1-\frac 1 m $ for some integers $m>1$. Our main result extends this when the coefficients of $\Dklt$ are no longer orbifold coefficients, but are any real numbers $a_i\ge 1/2$ (condition which is realized if $a_i$ is of orbifold type):

\begin{thma}
Let $X$ be a compact Kähler manifold and $\D=\sum a_i \D_i$ a $\R$-divisor with simple normal crossing support such that $K_X+\D$ is ample. We assume furthermore that the coefficients of $\D$ satisfy the following inequalities:
\[ 1/2 \le a_i \le 1. \]
Then $X\setminus \Supp(\D)$ carries a unique Kähler-Einstein metric $\omke$ with curvature $-1$ having mixed Poincaré and cone singularities along $\D$. 
\end{thma}

The conic case, \textit{ie} when the coefficients of $\D$ are stricly less than $1$), under the assumption that $K_X+\D$ is positive or zero, has been studied by R. Mazzeo \cite{Maz}, T. Jeffres \cite{Jef} and recently resolved independently by S. Brendle \cite{Brendle} and R. Mazzeo, T. Jeffres, Y. Rubinstein \cite{MJR} in the case of an (irreducible) smooth divisor, and by \cite{CGP} in the general case of a simple normal crossing divisor (having though all its coefficients greater than $\frac 1 2$). In the conic case where $K_X+\D<0$, some interesting existence results were obtained by R. Berman in \cite{rber} and T. Jeffres, R. Mazzeo and Y. Rubinstein in \cite{MJR}. Let us finally mention that in \cite{MJR}, it is proved that the potential of the Kähler-Einstein metric has polyhomogeneous expansion, which is much stronger than the assertion on the cone singularities of this metric.\\

Let us now give a sketch of the proof by detailing the organization of the paper.\\ 

The first step is, as usual, to relate the existence of Kähler-Einstein metrics to some particular Monge-Ampère equations. We explain this link in Proposition \ref{prop:ke}. The idea is that any negatively curved normalized Kähler-Einstein metric on $X_0$ with appropriate boundary conditions extends to a Kähler current of finite energy in $c_1(K_X+\D)$ satisfying on $X$ a Monge-Ampère equation of the type $\om_{\vp}^n =e^{\vp-\vp_{\D}}  \om^n$ where $\om$ is a Kähler form on $X$, and $\vp_{\D}=\sum a_i \log |s_i|^2+(\mathrm{smooth \, terms})$. One may observe that as soon as some $a_i$ equals $1$, the measure $e^{-\vp_{\D}}\om^n$ has infinite mass.

The uniqueness of the solution metric will then follow from the so-called comparison principle established by V.Guedj and A.Zeriahi for this special class of finite energy currents. \\

We are then reduced to solving some singular Monge-Ampère equation. The strategy consists in working on the open manifold $\Xlc:=X\setminus \Dlc$, and we are led to the following equation: $\om_{\vp}^n =e^{\vp-\vp_{\Dklt}}  \om^n$ where this time $\om$ is a Kähler form on $\Xlc$ with Poincaré singularities along $\Dlc$, and $\vp_{\Dklt}=\sum_{\{a_i<1\}} a_i \log |s_i|^2+(\mathrm{smooth \, terms})$. If $\vp_{\Dklt}$ were smooth, one could simply apply the results of Kobayashi and Tian-Yau. As it is not the case, we adapt the strategy of Campana-Guenancia-P\u{a}un to this setting:\\

 We start in section \ref{subs:approx} by regularizing $\vp_{\Dklt}$ into a smooth function (on $\Xlc$) $\vp_{\Dklt, \ep}$ and introducing smooth approximations $\ome$ of the cone metric on $\Xlc$ having Poincaré singularities along $\Dlc$.  Then we consider the regularized equation $\om_{\vpe}^n =e^{\vpe-\vp_{\Dklt,e}} \ome^n$ which we can solve for every $\ep>0$ (we are in the logarithmic case). The point is to construct our desired solution $\vp$ as the limit of $(\vpe)_{\ep}$; this is made possible by controlling (among other things) the curvature of $\ome$, and applying appropriate \textit{a priori} laplacian estimates which we briefly explain in section \ref{sec:estimates}. The final step is standard: it consists in invoking Evans-Krylov $\mathscr C^{2,\alpha}$ interior estimates, and concluding that $\vp$ is smooth on $X_0$ using Schauder estimates.\\ \\
 
In the last part of the paper, and as in \cite{CGP}, we try to use the Kähler-Einstein metric constructed in the previous sections to obtain the vanishing of some particular holomorphic tensors attached to a pair $(X,\D)$, $\D$ being still a $\R$-divisor with simple normal crossing support and having coefficients in $[0,1]$. This specific class consists in the holomorphic tensors which are the global sections of the locally free sheaf $T^r_s(X|\D)$ introduced by Campana in \cite{Camp1}: they are holomorphic tensors with prescribed zeros or poles along $\D$. Thanks to their realization as bounded tensors with respect to some (or equivalently, any) twisted metric $g$ with mixed cone and Poincaré singularities along $\D$, given in Proposition \ref{prop:bdo} we can use Theorem A to prove the following:

\begin{thmb}
Let $(X,\D)$ be a pair satisfying the assumptions of Theorem A. Then, there is no non-zero holomorphic tensor of type $(r,s)$ whenever $r\ge s+1$:
\[ H ^0(X,T^{r}_s(X|\D))=0.\]
\end{thmb}

The proof of this results follows closely the one of its analogue in \cite{CGP}: we use a Bochner formula applied to the truncated holomorphic tensors, and the key point is to control the error term. However, a new difficulty pops up here, namely we have to deal with an additional term coming from the curvature of the line bundle $\Ox(\lfloor \D \rfloor)$; fortunately, it has the right sign.

\section*{Acknowledgments}

I am very grateful to Sébastien Boucksom for his patient and careful reading of the preliminary versions, and his several highly valuable comments and suggestions to improve both the organization and the content of this paper.

\noindent
I would like to also thank warmly Mihai P\u aun for the precious help he gave me to elaborate the last section of this article.

\section{Preliminaries}

In this first section devoted to the preliminaries, we intend to fix the notations and the scope of this paper. We also recall some useful objects introduced in \cite{KobR} and \cite{Tia} within the framework of the logarithmic case; finally, we explain briefly some a priori estimates which are going to be some of our main tools in the proof of the main theorem.

\subsection{Notations and definitions}

All along this work, $X$ will be a compact Kähler manifold of complex dimension $n$. We will consider effective $\R$-divisors $\D=\sum a_i \D_i$ with simple normal crossing support, and such that their coefficients $a_i$ belong to $[0,1]$.

\noindent
It will be practical to separate the hypersurfaces $\D_i$ appearing with coefficient $1$ in $\D$ from the other ones. For this, we write:
\begin{eqnarray*}
\D&=&\sum_{\{a_i<1\}} a_i \D_i + \sum_{\{a_i=1\}} \D_i\\
&=& \Dklt+\Dlc
\end{eqnarray*}

\noindent
These notations come from the framework of the pairs in birational geometry; klt stands for \textit{Kawamata log-terminal} whereas lc means \textit{log-canonical}. In this language, $(X,\D)$ is called a log-smooth lc pair, and $(X,\Dklt)$ is a log-smooth klt pair. Apart from these practical notations, we will not use this terminology.\\

We will denote by $s_i$ a section of $\Ox(\D_i)$ whose zero locus is the (smooth) hypersurface $\D_i$, and, omitting the dependance in the metric, we write $\Theta(\D_i)$ the curvature form of $(\Ox(\D_i),h_i)$ for some hermitian metric on $\Ox(\D_i)$. Up to scaling the $h_i$'s, one can assume that $|s_i| \le e^{-1}$, and we will make this assumption all along the paper. Finally, we set $X_0:=X\setminus \Supp(\D)$ and $\Xlc:=X\setminus \Supp(\Dlc)$.\\

In the introduction, we introduced a natural class of growth of Kähler metrics near the divisor $\D$ which we called metrics with mixed Poincaré and cone singularities along $\D$. They are the Kähler metrics locally equivalent to the model metric $\om_{\rm mod}= \sum_{j=1}^r \frac{i dz_j\wedge d\bar z_j}{|z_j|^{2d_j}} + \sum_{j=r+1}^s \frac{i dz_j\wedge d\bar z_j}{|z_j|^{2} \log^{2}|z_j|^{2}}+\sum_{j=r+s+1}^n i dz_j\wedge d\bar z_j$ whenever the pair $(X,\D)$ is locally isomorphic to $(X_{\rm mod}, D_{\rm mod})$ with $X_{\rm mod}=(\mathbb{D}^*)^r\times (\mathbb{D}^*)^s \times \mathbb{D}^{n-(s+r)}$ and $D_{\rm mod}=d_1 [z_1=0]+\cdots+d_r [z_r=0]+[z_{r+1}=0]+\cdots + [z_{r+s}=0]$, with $d_i<1$.

\noindent
The following elementary lemma ensures that given a pair $(X,\D)$ as above, Kähler metrics with mixed Poincaré and cone singularities along $\D$ always exist:

\begin{lemm} 
\phantomsection
\label{lem:cg}
The following $(1,1)$-form
\[\om_{\D}:= \om_0+\sum_{\{a_i<1\}} \ddc |s_i|^{2(1-a_i)}-\sum_{\{a_i=1\}} \ddc \log \log \frac{1}{|s_i|^2} \]
defines a Kähler form on $X_0$ as soon as $\om_0$ is a sufficiently positive Kähler metric on $X$. Moreover, it has mixed Poincaré and cone singularities along $\D$. 
\end{lemm}

\begin{proof}
This can be seen by a simple computation: combine e.g. \cite[Proposition 2.1]{Clodo} with \cite[Proposition 2.1]{CG} or \cite[Proposition 2.17]{Gri}.
\end{proof}

Before we end this paragraph, we would like to emphasize the different role played by the $\D_i$'s whether they appear in $\D$ with coefficient $1$ or stricly less than $1$. Here is some explanation: let $0<\alpha<1$ be a real number, and $\om_{\alpha}=\frac{(1-\alpha)^2 i dz \wedge d\bar z}{|z|^{2\alpha}(1-|z|^{2(1-\alpha)})^2}$; its curvature is constant equal to $-1$ on the punctured disc $\mathbb{D}^*$, and it has a cone singularity along  the divisor $\alpha [0]$. Then, when $\alpha$ goes to $1$, $\omega_{\alpha}$ converges pointwise to the Poincaré metric $\omega_P=\frac{i dz \wedge d\bar z}{|z|^2 \log^2 |z|^2}$.\\\\

\emph{In the following, any pair $(X,\D)$ will be implicitely assumed to be composed of a compact Kähler manifold $X$ and a $\R$-divisor $\D$ on $X$ having simple normal crossing support and coefficients belonging to $[0,1]$.}\\

\subsection{Kähler-Einstein metrics for pairs}
\label{sec:kep}

As explained in the introduction, the goal of this paper is to find a Kähler metric on $X_0$ with constant Ricci curvature, and having mixed Poincaré and cone singularities along the given divisor $\D$. The second condition is essential and as important as the first one; the proof of the vanishing theorem for holomorphic tensors in the last section will render an account of this and shall surely convince the reader. Let us state properly the definition:

\begin{defi}
\phantomsection
\label{def:ke}
A Kähler-Einstein metric for a pair $(X,\D)$ is defined to be a Kähler metric $\om$ on $X_0$ satisfying the following properties:
\begin{enumerate}
\item[$\bullet$] $\Ric \om = \mu \, \om\, \,$  for some real number $\mu$;
\item[$\bullet$] $\om$ has mixed Poincaré and cone singularities along $\D$.
\end{enumerate}
\end{defi}

\begin{rema}
\phantomsection
\label{rem:bm}
Unlike cone singularities, Poincaré singularities are intrinsically related to negative curvature geometry:
\begin{enumerate}
\item[$\cdot$]
The Bonnet-Myers Theorem  tells us that in the case where $\Dklt=0$ (so that we work with complete metrics), there cannot exist Kähler-Einstein metrics in the previous sense with $\mu > 0$. However, if $\Dlc=0$, there may exist Kähler-Einstein metrics with positive curvature, and the question of their existence is often a difficult question (see e.g. \cite{BBEGZ} or \cite{rber}).
\item[$\cdot$]
As for the Ricci-flat case ($\mu=0$), it also has to be excluded. Indeed, there cannot be any Ricci-flat metric on the punctured disc $\mathbb{D}^*$ with Poincaré singularity at $0$; to see this, we write $\om = \frac{i}{2}e^{2u}dz \wedge d\bar z$ such a metric, and then $u$ has to satisfy the following properties: $u$ is harmonic on $\mathbb{D}^*$ and $e^{2u}$ behaves like $\frac 1 {|z|^2 \log^2 |z|^2}$ near $0$, up to constants. But it is well-known that any harmonic function $u$ on $\D^*$ can be written $u=\mathrm{Re}(f)+c \log |z|$ for some holomorphic function $f$ on $\mathbb{D}^*$ and some constant $c\in \R$. Clearly, $f$ cannot have an essential singularity at $0$; moreover, because of the logarithmic term in the Poincaré metric, $f$ can neither be bounded, nor have a pole at $0$. This ends to show that in general (and for local reasons), there does not exist Ricci-flat Kähler-Einstein metric in the sense of the previous definition (whenever $\Dlc \neq 0$).
\end{enumerate}
For these reasons, we will focus in the following on the case of negative curvature, which we will normalize in $\mu =-1$.
\end{rema}

\subsection{The logarithmic case}
\label{sec:log}
For the sake of completeness, we will briefly recall in this section the proof of the main result (Theorem \ref{main}) in the logarithmic case, namely when $\D=\Dlc$, \textit{ie} when $\Dklt=0$. As we already explained, this was achieved by Kobayashi \cite{KobR} and Tian-Yau \cite{Tia} in a very similar way. In this section, we will assume that $(X,\D)$ is logarithmic, so that $X_0=\Xlc$. \\

We will use the following terminology which is convenient for the following:

\begin{defi}
\phantomsection
\label{def:cg}
We say that a Kähler metric $\om$ on $X_0$ is of Carlson-Griffiths type if there exists a Kähler form $\om_0$ on $X$ such that 
$\om=\om_0-\sum_K \ddc \log \log \frac 1{|s_k|^{2}}$.
\end{defi}

As observed in Lemma \ref{lem:cg}, such a metric always exists, and it has Poincaré singularities along $\D$. In \cite{CG}, Carlson and Griffiths introduced such a metric for some $\om_0 \in c_1(K_X+\D)$. The reason why we exhibit this particular class of Kähler metric on $X_0$ having Poincaré singularities along $\D$ is that we have an exact knowledge on its behaviour along $\D$, much more precise that its membership of the previously cited class. For example, Lemma \ref{lem:cka} mirrors this fact.\\

We start from a compact Kähler manifold $X$ with a simple normal crossing divisor $\D=\sum \D_k$ such that $K_X+\D$ is ample. We want to find a Kähler metric $\omke$ on $X_0=X\setminus \D$ with $-\Ric \omke = \omke$, and having Poincaré singularities along $\D$. If we temporarily forget the boundary condition, the problem amounts to solve the following Monge-Ampère equation on $X_0$:
\[(\om+\ddc \vp)^n = e^{\vp+F}\om^n\]
where $\om$ is a Kähler metric on $X_0$ of Carlson-Griffiths type (cf. Definition \ref{def:cg}), and $F= -\log \left( \prod |s_k|^2 \log^2 |s_k|^2 \cdot \om^n/\om_0^n \right)+(\mathrm{smooth \, terms \, on\,} X$) for some Kähler metric $\om_0$ on $X$. \\

The key point is that $(X_0,\om)$ has bounded geometry at any order. Let us get a bit more into the details. To simplify the notations, we will assume that $\D$ is irreducible, so that locally near a point of $\D$, $X_0$ is biholomorphic to $\mathbb{D}^* \times \mathbb{D}^{n-1}$, where $\mathbb{D}$ (resp. $\mathbb{D}^*$) is the unit disc (resp. punctured disc) of $\mathbb C$. We want to show that, roughly speaking, the components of $\om$ in some appropriate coordinates have bounded derivatives at any order. The right way to formalize it consists in introducing quasi-coordinates: they are maps from an open subset $V\subset \mathbb C^n$ to $X_0$ having maximal rank everywhere. So they are just locally invertible, but these maps are not injective in general. \\
To construct such quasi-coordinates on $X_0$, we start from the univeral covering map $\pi:\mathbb{D}\to\mathbb{D}^*$, given by $\pi(w)=e^{\frac{w+1}{w-1}}$. Formally, it sends $1$ to $0$. The idea is to restrict $\pi$ to some fixed ball $B(0,R)$ with $1/2<R<1$, and compose it (at the source) with a biholomorphism $\Phi_{\eta}$ of $\mathbb{D}$ sending $0$ to $\eta$, where $\eta$ is a real parameter which we will take close to $1$. If want to write a formula, we set $\Phi_{\eta}(w)=\frac{w+\eta}{1+\eta w}$, so that the quasi-coordinate maps are given by
$\Psi_{\eta}=\pi\circ \Phi_{\eta}\times \mathrm{Id}_{\mathbb{D}^{n-1}}:V=B(0,R)\times \mathbb{D}^{n-1}\to \mathbb{D}^*$, with $\Psi_{\eta}(v,v_2, \ldots, v_n)=(e^{\frac{1+\eta}{1-\eta} \frac{v+1}{v-1}}, v_2, \ldots, v_n)$.\\
Once we have said this, it is easy to see that $X_0$ is covered by the images $\Psi_{\eta}(V)$ when $\eta$ goes to $1$, and for all the trivializing charts for $X$, which are in finite number. Now, an easy computation shows that the derivatives of the components of $\om$ with respect to the $v_i$'s are bounded uniformly in $\eta$. This can be thought as a consequence of the fact that the Poincaré metric is invariant by any biholomorphism of the disc.\\

At this point, it is natural to introduce the Hölder space of $\cka$-functions on $X_0$ using the previously introduced quasi-coordinates:
\begin{defi}
\phantomsection
\label{def:cka}
For a non-negative integer $k$, a real number $\alpha \in ]0.1[$, we define: 
\[\cka(X_0)=\{u\in \mathscr C^k (X_0);\,\, \sup_{V, \eta} ||u\circ \Psi_{\eta} ||_{k,\alpha}<+\infty\}\]
where the supremum is taken over all our quasi-coordinate maps $V$ (which cover $X_0$). Here $||\cdot ||_{k,\alpha}$ denotes the standard $\cka$-norm for functions defined on a open subset of $\mathbb C^n$.
\end{defi}

The following fact, though easy, is very important for our matter:
\begin{lemm}
\phantomsection
\label{lem:cka}
Let $\om$ be a Carlson-Griffiths type metric on $X_0$, and $\om_0$ some Kähler metric on $X$. Then 
\[F_0:=\log \left(\prod |s_k|^2 \log^2 |s_k|^2\cdot  \om^n /\om_0^n \right)\]
belongs to the space $\cka(X_0)$ for every $k$ and  $\alpha$.
\end{lemm}

\begin{proof}
The first remark is that $F_0$ is bounded (cf. \cite[Lemma 1.(ii)]{KobR} or the beginning of section \ref{subs:lb}), and $F_0\in\cka(X_0)$ if and only if $e^{F_0}\in\cka(X_0)$. So in the following, we will deal with $e^{F_0}$. \\
Then, as the (elementary) computations of Lemma \ref{lem:calc} show, it is enough to check that the functions on $\mathbb{D}^*$ (say with radius 1/2) defined by $z\mapsto \frac{1}{\log |z|^2}, z \mapsto |z|^2 \log |z|^2$ and $z \mapsto |z|^2 \log^2 |z|^2$ are in $\cka(\mathbb{D}^*)$. But in the quasi-coordinates given by $\Phi_{\eta}$, $\frac{1}{\log |z|^2}= \frac 1 2\cdot\frac{1-\eta}{1+\eta} \frac{|v|^2-1}{|v-1|^2}$ and $|z|^2 \log^{\alpha} |z|^2=
\left(\frac 1 2\cdot\frac{1+\eta}{1-\eta} \frac{|v-1|^2}{|v|^2-1}\right)^{\alpha} e^{2\cdot\frac{1+\eta}{1-\eta} \frac{|v|^2-1}{|v-1|^2}}$, for $v\in B(0,R)$ with $R<1$, and where $\alpha \in \R$. Now there is no difficulty in seeing that these two functions of $v$ are bounded when $\eta$ goes to $1$ (actually this property does not depend on the chosen coordinates), and so are their derivatives (still with respect to $v$); this is obvious for the first function, and for the second one, it relies on the fact that $x^m e^{-x}$ goes to $0$ as $x\to +\infty$, for all $m\in \Z$.
\end{proof}

The end of the proof consists in showing that the Monge-Ampère equation $(\om+\ddc \vp)^n = e^{\vp+f}\om^n$ has a unique solution $\vp\in \cka(X_0)$ for all functions $f\in \cka(X_0)$ with $k\ge 3$. This can be done using the continuity method in the quasi-coordinates. In particular, applying this to $f=F$ (cf beginning of the section), which the previous lemma allows to do, this will prove the existence of a negatively curved Kähler-Einstein metric, which is equivalent to $\om$ (in the \emph{strong} sense: $\vp\in \cka(X_0)$ for all $k,\alpha$).\\

To summarize, the theorem of Kobayashi and Tian-Yau is the following:

\begin{theo}
\phantomsection
\label{thm:kob}
Let $(X,\D)$ be a logarithmic pair, $\om$ a Kähler form of Carlson-Griffiths type on $X_0$, and $F\in \cka(X_0)$ for some $k\ge 3$. Then there exists $\vp \in \cka(X_0)$ solution to the following equation: 
\[(\om+\ddc \vp)^n = e^{\vp+F}\om^n\]
In particular if $K_X+\D$ is ample, then there exists a (unique) Kähler-Einstein metric of curvature $-1$ equivalent to $\om$.
\end{theo}

\subsection{\textit{A priori} estimates}
\label{sec:estimates}

In this section, we recall the classical estimates valid for a large class of complete Kähler manifolds; they are derived from the classical estimates over compact manifolds using the generalized maximum principle of Yau \cite{Yau2}.  We will use them in an essential manner in the course of the proof of our main theorem. Indeed, our proof is based upon a regularization process, and in order to guarantee the existence of the limiting object, we need to have a control on the $\mathscr C ^k$ norms.

\begin{theo}
\phantomsection
\label{thm:max}
Let $X$ be a complete Riemannian manifold with Ricci curvature bounded from below. Let $f$ be a $\mathscr C^2$ function which is bounded from below on $M$. Then for every $\ep>0$, there exists $x\in X$ such that at $x$, 
\[|\nabla f| < \ep, \quad \Delta f >-\ep, \quad f(x)< \inf_X f +\ep. \]
\end{theo}

\noindent
From this, we easily deduce the following result, stated in \cite[Proposition 4.1]{CY}.
\begin{prop}
\phantomsection
\label{cy1}
Let $(X,\om)$ be a $n$-dimensional complete Kähler manifold, and $F\in \mathscr C^{2}(X)$ a bounded function. We assume that we are given $u\in \mathscr C^{2}(X)$ satisfying $\om+\ddc u>0$ and
\begin{equation}
\label{eqn:cy}
(\omega+\ddc u)^{n} = e^{u+F}\omega^{n} \nonumber
\end{equation}
Suppose that the bisectional curvature of $(X,\om)$ is bounded below by some constant, and that $u$ is a bounded function. Then 
\[\sup_X |u| \le \sup_X |F|.\]
\end{prop}

We emphasize the fact that the previous estimate does not depend on the lower bound for the bisectional curvature of $(X,\om)$.

\noindent
As for the Laplacian estimate, we have the following (we could also have used \cite[Proposition 4.2]{CY}):

\begin{prop}
\phantomsection
\label{cy2}
Suppose that the bisectional curvature of $(X,\om)$ is bounded below by some constant $-B, B>0$, and that $u$ as well as its Laplacian $\D u$ are bounded functions on $X$. If $\om+\ddc u$ defines a complete Kähler metric on $X$ with Ricci curvature bounded from below, then 
\[\sup_X \, (n+\D u) \le C\]
where $C>0$ only depends on $\sup |F|$, $\inf \D F$, $B$ and $n$.
\end{prop}

\begin{proof}[Sketch of the proof]
We set $\om'=\om+\ddc u$, and $\D'$ is defined to be the Laplacian with respect to $\om'$.

\noindent
Using \cite[Lemma 2.2]{CGP}, we obtain $\D' (\tr_{\om}\om') \ge \frac{\D F}{\tr_{\om'}\om}-B \tr_{\om'}\om$, and from this we may deduce that 
\[\D' (\tr_{\om}\om'-(C_1+1)u)\ge \tr_{\om'}\om-C_2\]
where $C_1,C_2$ are constant depending only $B$, $\inf \D F$ and $n$. The assumptions allow us to use the generalized maximum principle stated as Theorem \ref{thm:max}
to show that $\sup \tr_{\om'}\om \le C_3$. As $\om'=e^{F+u}\om$, and as we have at our disposal uniform estimates on $\sup |u|$ thanks to \ref{cy1}, the usual arguments work here to give a uniform bound $\sup \, (n+\D u) \le C$. We refer e.g. to \cite[section 2]{CGP} for more details. 
\end{proof}

\section{Uniqueness of the Kähler-Einstein metric}
\label{sec:kema}

In this section, we begin to investigate the questions raised in the introduction concerning the existence of Kähler-Einstein metrics for pairs $(X,\D)$. The first thing to do is, as usual, to relate the existence of theses metrics to the existence of solutions for some Monge-Ampère equations. We will be in a singular case, so we have to specify the class of $\om$-psh functions to which we are going to apply the Monge-Ampère operators. This is the aim of the few following lines, where we will recall some recent (but relatively basic) results of pluripotential theory. We refer to \cite{GZ2} or \cite{BEGZ} for a detailed treatment. \\

\subsection{Energy classes for quasi-psh functions}

Let $\om$ be a Kähler metric on $X$; the class $\mathcal E(X,\om)$ is defined to be composed of $\om$-psh functions $\vp$ such that their non-pluripolar Monge-Ampère $(\om + \ddc \vp)^n$  has full mass $\int_X \om^n$ (cf. \cite{GZ2}, \cite{BEGZ}). An alternate way to apprehend those functions is to see them as the largest class where one can define $(\om+\ddc \vp)^n $ as a measure which does not charge pluripolar sets. Those functions satisfy the so-called comparison principle, which we are going to use in an essential manner for the uniqueness of our Kähler-Einstein metric:

\begin{prop}[Comparison Principle, \cite{GZ2}]
\phantomsection
\label{GZ}
Let $\vp,\psi\in \mathcal E(X,\om)$. Then we have:
\[ \int_{\{\vp<\psi\}} (\om + \ddc \psi)^n \le \int_{\{\vp<\psi\}} (\om + \ddc \vp)^n.\]
\end{prop}

An important subset of $\mathcal E(X,\om)$ is the class $\mathcal E^1(X,\om)$ of functions in the class $\mathcal E(X,\om)$ having finite $\mathcal E^1$-energy, namely $\mathcal E^1(\vp):=\int_X |\vp| (\om + \ddc \vp)^n <+\infty$. Every smooth (or even bounded) $\om$-psh function belongs to this class.

\noindent
In order to state an useful result for us, we recall the notion of \emph{capacity} attached to a compact Kähler manifold $(X,\om)$, as introduced in \cite{GZ}, generalizing the usual capacity of Bedford-Taylor (\cite{BT82}): for every Borel subset $K$ of $X$, we set: 
\[\mathrm{Cap}_{\om}(K):=\sup\left\{ \int_K \om_{\vp}^n; \, \, \vp \in \mathrm{PSH}(X,\om), \, 0 \le \vp \le 1\right\} \]
\noindent
There is an useful criteria to show that some $\om$-psh function belongs to the class $\mathcal E^1(X,\om)$ without checking that it has full Monge-Ampère mass, but only using the capacity decay of the sublevel sets. It appears in different papers, among which \cite[Lemma 5.1]{GZ2}, \cite[Proposition 2.2]{BGZ}, \cite[Lemma 2.9]{BBGZ}:
\begin{lemm}
\phantomsection
\label{lem:cap}
Let $\vp \in \mathrm{PSH}(X,\om)$. If 
\[\int_{t=0}^{+\infty}  t^n\,  \mathrm{Cap}_{\om} \{\vp < -t\}\,dt < +\infty\]
then $\vp \in \mathcal E^1(X,\om)$.
\end{lemm}

Now we have enough background about these objects to state and prove the result we will use in the next section. Let us first fix the notations.

\noindent
Let $(X,\om_0)$ be a Kähler manifold, and $\D=\sum_{k\in K} \D_k$ a simple normal crossing divisor. We choose sections $s_k$ of $\Ox(\D_k)$ whose divisor is precisely $\D_k$, and we fix some smooth hermitian metrics on those line bundles. We can assume that $|s_k| \le e^{-1}$, and we know that, up to scaling the metrics, one may assume that $\om_0- \sum_k \ddc \log \log \frac{1}{|s_k|^2}$ is positive on $X_0$, and defines a Kähler current on $X$.
\begin{prop}
\phantomsection
\label{prop:loglog}
The function
\[\vp_0 = - \sum_{k\in K} \log \log \frac{1}{|s_k|^2}\]
belongs to the class $\mathcal E^1(X,\om_0)$.
\end{prop}

\begin{proof}
We want to apply Lemma \ref{lem:cap}. To compute the global capacity as defined above, or at least know the capacity decay of the sublevel sets, it is convenient to use the Bedford-Taylor capacity.
But a result due to Ko\l odziej \cite{Kol} (see also \cite[Proposition 2.10]{GZ}), states that up to universal multiplicative constants, the capacity can be computed by the local Bedford-Taylor capacities on the trivializing charts of $X$.\\

Therefore, we are led to bound from above $\mathrm{Cap}_{BT}\{u<-t\}$ in the unit polydisc of $\mathbb C^n$, where $u=\sum_{i=1}^p -\log(-\log |z_i|^2)$ for some  $p\le n$. 
As \[\{u<-t\} \subset \bigcup_{i=1}^p \left\{-\log(-\log|z_i|^2)<-\frac t p\right\}\] one can now assume that $p=1$. But $\mathrm{Cap}_{BT}\{\log|z|^2<-t)=2/t$ (see e.g \cite[Example 13.10]{DemP}), whence $\mathrm{Cap}_{BT}\{-\log(-\log|z_i|^2)<-t)=2e^{-t}$. The result follows.
\end{proof}

\begin{rema}
An alternate way to proceed is to show that the smooth approximations $\vpe:=-\sum_{k\in K} \log \log \frac{1}{|s_k|^2+\ep^2}$ of $\vp_0$ have (uniformly) bounded $\mathcal E^1$-energy, which also allows to conclude that $\vp_0 \in \mathcal E^1(X,\om_0)$ thanks to \cite[Proposition 2.10 \& 2.11]{BEGZ}.
\end{rema}

\subsection{From Kähler-Einstein metrics to Monge-Ampère equations}

The following proposition explains how to relate Kähler-Einstein metrics for a pair $(X,\D)$ and some Monge-Ampère equations, the difficulty being here that we have to deal with singular weights/potentials for which the definitions and properties of the Monge-Ampère operators are more complicated than in the smooth case. Note that this result generalizes \cite[Proposition 5.1]{rber}:

\begin{prop}
\phantomsection
\label{prop:ke}
Let $X$ be a compact Kähler manifold, and $\D=\sum a_j \D_j$ an effective $\R$-divisor with simple normal crossing support, such that $a_j\le 1$ for all $j$. We assume that $K_X+\D$ is ample, and we choose a Kähler metric $\om_0\in c_1(K_X+\D)$. Then any Kähler metric $\om$ on $X_0$ satisfying: 
\begin{enumerate}
\item[$\bullet$] $-\Ric \om = \om$ on $\xmd$;
\item[$\bullet$] There exists $C>0$ such that: \[\quad C^{-1} \om^n \le \frac{\om_0^n}{\prod_{\{a_i<1\}}|s_i|^{2a_i}\prod_{\{a_i=1\} }|s_i|^2 \log^2 |s_i|^2} \le  C \om^n\]
\end{enumerate}
extends to a Kähler current $\om=\om_0+\ddc \vp$ on $X$ where $\vp \in \mathcal E^1(X,\om_0)$ is a solution of 
\[(\om_0+\ddc \vp)^n = e^{\vp-\vp_{\D}}\om_0^n\]
and $\vp_{\D}=\sum _{r \in J\cup K} a_{r} \log |s_{r}|^2+f$ for some $f\in \mathscr C^{\infty}(X)$. Furthermore there exists at most one such metric $\om$ on $\xmd$.
\end{prop}

\begin{rema}
One can observe that although $e^{\vp-\vp_{\D}} \om_0^n$ has finite mass, $e^{-\vp_{\D}} \om_0^n$ does not (as soon as $\Dlc \neq 0)$.
\end{rema}

\begin{proof}
We recall that $\Theta(\D_i)$ denotes the curvature of $(\Ox(\D_i),h_i)$, and we write $\Theta(\Dklt)=\sum_{\{a_i<1\}} a_i \Theta(\D_i)$, $\Theta(\Dlc)=\sum_{\{a_i=1\}}  \Theta(\D_i)$ and $\Theta(\D)=\Theta(\Dklt)+\Theta(\Dlc)$. All those forms are smooth on $X$.

\noindent
Let us define a smooth function $\psi$ on $\xmd$ by:
 \[\psi_0:=\log \left( \frac{\prod_{j\in J} |s_j|^{2a_j} \prod_{k\in K} |s_k|^2 \log^2 |s_k|^2 \, \om^n}{\om_0^n} \right)\]
By assumption, $\psi_0$ is bounded on $\xmd$, so that $\psi:=\psi_0- \sum_k \log \log^2 \frac{1}{|s_k|^2}$ is bounded above on $\xmd$. On this set, we have
\[\ddc \psi = \om+\Ric \om_0^n  + \Theta(\D)\]
so that $\psi$ is $M\om_0$-psh for some $M>0$ big enough. As it is bounded above, it extends to a (unique) $M\om_0$-psh function on the whole $X$, which we will also denote by $\psi$. 
Let now $f$ be a smooth potential on $X$ of $\Ric \om_0^n +\om_0-\Theta(\D)$. It is easily shown that $\vp:=\psi-f$ satisfies $\om_0+\ddc \vp = \om$ on $\xmd$. \\

\noindent
From the definition of $\vp$, we see that $\vp = 2 \vp_0 + \mathcal O(1)$, where $\vp_0= - \sum_{k\in K} \log \log \frac{1}{|s_k|^2}$. Therefore, Proposition \ref{prop:loglog} ensures that $\vp \in \mathcal E^1(X,\om_0)$, so that its non-pluripolar Monge-Ampère $(\om_0+\ddc \vp)^n$ satisfies the equation
\begin{eqnarray*}
(\om_0+\ddc \vp)^n&=&\frac{e^{\vp-f}\om_0^n}{\prod_{r\in J\cup K} |s_{r}|^{2a_r}}\\
&=&e^{\vp-\vp_{\Delta}}\om_0^n
\end{eqnarray*}
on the whole $X$, with the notations of the statement. By the comparison principle (Proposition \ref{GZ}), if the previous equation had two solutions $\vp,\psi\in \mathcal E^1(X,\om_0)$, then on the set $A=\{\vp < \psi\}$, we would have 
\[\int_A e^{\psi-\vp_{\Delta}}\om_0^n \le \int_A e^{\vp-\vp_{\Delta}}\om_0^n \]
but on $A$, $e^{\psi}>e^{\vp}$ so that $A$ has zero measure with repect to the measure $e^{-\vp_{\Delta}}\om_0^n$, so it has zero measure with respect to $\om_0^n$. We can do the same for $B=\{\psi < \vp\}$, so that $\{\vp=\psi\}$ has full measure with respect to $\om_0^n$. As $\vp, \psi$ are $\om_0$-psh, they are determined by their data almost everywhere, so they are equal on $X$. This finishes to conclude that our $\vp$ is unique, so that the proposition is proved.
\end{proof}

\begin{rema}
In the logarithmic case ($\D=\Dlc$), the metrics at stake are complete, so that their uniqueness follow from the generalized maximum principle of Yau (cf. \cite{KobR}, \cite{Tia} e.g). In the conic case, Ko\l odziej's theorem \cite{Kolo} ensures that the potentials we are dealing with are continuous, and the unicity follows from the classical comparison principle established in \cite[Theorem 4.1]{BT82}.
\end{rema}

As Kähler metrics with mixed Poincaré and cone singularities clearly satisfy the second condition of the proposition, we deduce that any negatively curved normalized Kähler-Einstein metric must be obtained by solving the global equation $(\om_0+\ddc \vp)^n = e^{\vp-\vp_{\D}}\om_0^n$ on $X$, for $\vp \in \mathcal{E}^1(X,\om_0)$, and $\vp_{\D}=\sum _{r \in J\cup K} a_{r} \log |s_{r}|^2+f$ for some $f\in \mathscr C^{\infty}(X)$. 
We will now show how to solve the previous equation, and derive from this the existence of negatively curved Kähler-Einstein metrics and their zero-th order asymptotic along $\D$.

\section{Statement of the main result}

Here is a result which encompasses the previous results of \cite{CGP}, Kobayashi (\cite{KobR}) and Tian-Yau (\cite{Tia}). This provides a (positive) partial answer to a question raised in \cite[section 10]{CGP}. 

\begin{theo}
\phantomsection
\label{main}
Let $X$ be a compact Kähler manifold, and $\D=\sum a_i \D_i$ an effective $\R$-divisor with simple normal crossing support such that its coefficients satisfy the inequalities:
\[ 1/2 \le a_i \le 1. \]
Then for any Kähler form $\omega$ on $\Xlc$ of Carlson-Griffiths type and any function $f\in \cka(\Xlc)$ with $k\ge 3$, there exists a Kähler metric $\omi=\om+\ddc \vp$ on $X_0$ solution to the following equation:
\[ (\om+\ddc \vp)^{n} =\frac{e^{\vp+f}}{\prod_{\{a_i<1\}}|s_i|^{2a_i}}\,\om^{n}\]
such that $\omi$ has mixed Poincaré and cone singularities along $\D$.
\end{theo}

We refer to section \ref{sec:log} and more precisely to Definition \ref{def:cka} for the definition of the space $\cka(\Xlc)$; one important class of functions belonging to $\cka(\Xlc)$ is pointed out in Lemma \ref{lem:cka}, and we will use it for proving the following result.

\begin{coro}
Let $(X,\D)$ be a pair such that $\D=\sum a_i \D_i$ is a divisor with simple normal crossing support whose coefficients satisfy the inequalities
\[ 1/2 \le a_i \le 1. \]
If $K_X+\D$ is ample, then $X_0$ carries a unique Kähler-Einstein metric $\omke$ of curvature $-1$ having mixed Poincaré and cone singularities along $\D$. 
\end{coro}

Here, by ample, we mean that $c_1(K_X+\D)$ contains a Kähler metric, or equivalently that $K_X+\D$ is a positive combination of ample $\Q$-divisors.

\begin{proof}
We choose $(h_i)$ and $h_{K_X}$ some smooth hermitian metrics on the line bundles $\Ox(\D_i)$ and $\Ox(K_X)$ respectively such that the product metric $h$ on $K_X+\D$ has positive curvature $\om_0$, and up to renormalizing the metrics $h_k$, one can assume that $\om:=\om_0-\sum_{\{a_k=1\}}\ddc \log \log \frac{1}{|s_k|^2}$ defines a Kähler metric on $\Xlc$ with Poincaré singularities along $\Dlc$; more precisely it is of Carlson-Griffiths type. \\
Lemma \ref{lem:cka} shows that one can write
\[\om^n = \frac{e^{-B}\Psi}{\prod |s_k|^2 \log^2 |s_k|^2} \]
with $\Psi$ the smooth volume form on $X$ attached to $h_{K_X}$ (in particular $-\Ric \Psi = \Theta_{h_{K_X}}(K_X)$, the curvature of $(\Ox(K_X), h_{K_X})$), and $B \in \cka(X\setminus \Dlc)$ for all $k$ and $\alpha$. \\
Now we use Theorem \ref{main} with $f=B$, and $\om$ as reference metric. We then get a Kähler metric $\omke := \om+\ddc \vp$ on $X\setminus \Supp(\D)$ with mixed Poincaré and cone singularities along $\D$ satisfying
\[ (\om+\ddc \vp)^{n} =\frac{e^{\vp+B}}{\prod_{j\in J}|s_j|^{2a_j}}\,\om^{n}. \]
Therefore, 
\begin{eqnarray*}
- \Ric (\omke) &=& \ddc(\vp+B)-\ddc B + \Theta_{h_{K_X}}(K_X) - \sum_{k\in K} \left(\ddc \log |s_k|^2 - \ddc \log \log \frac{1}{|s_k|^2}\right)\\
&& - \sum_{j\in J} \ddc \log |s_j|^{2a_j} \\
&=&\ddc \vp + \Theta(K_X)+\Theta(\Dlc)+\Theta(\Dklt) - \sum_{k\in K}  \ddc \log \log \frac{1}{|s_k|^2}\\
&=& \omke.
\end{eqnarray*}
Moreover, $\omke$ has mixed Poincaré and cone singularities along $\D$, so it is a Kähler-Einstein metric for the pair $(X,\D)$.\\
As for the uniqueness of $\omke$, it follows directly from Proposition \ref{prop:ke}.
\end{proof}

\section{Proof of the main result}

As we explained in the introduction, the natural strategy is to combine the approaches of \cite{CGP} and Kobayashi (\cite{KobR}). More precisely we will produce a sequence of Kähler metrics $(\ome)_{\ep}$ on $X\setminus \D_{lc}$ having Poincaré singularities along $\D_{lc}$ and acquiring cone singularities along $\D_{klt}$ at the end of the process when  $\ep =0$. \\

\subsection{The approximation process}
\label{subs:approx}
We keep the notation of Theorem \ref{main}, so that $\om$ is a Kähler form on $\Xlc$ of Carlson-Griffiths type; in particular it has Poincaré singularities along $\Dlc$.\\
We define, for any sufficiently small $\ep>0$, a Kähler form $\ome$ on $\Xlc$ by
\[ \ome:=\om+\ddc \psi_{\ep} \]
where $\psi_{\ep}=\frac{1}{N}\sum_{\{a_j<1\}} \chi_{j,\ep}(\ep^{2}+|s_j|^{2})$  for $\chi_{j,\ep}$ functions defined by:
\[\chi_{j,\ep}(\varepsilon^2+ t)= \frac 1{\tau_j}\int_0^t{\frac{(\varepsilon^2+ r)^{\tau_j}- \varepsilon^{2\tau_j}} r}dr\]
for any $t\ge 0$. The important facts to remember about this construction are the following ones, extracted from \cite[section 3]{CGP}:
\begin{enumerate}
\item[$\cdot$] For $N$ big enough, $\ome$ dominates (as a current) a Kähler form on $X$ because $\om$ already does;
\item[$\cdot$] $\psi_{\ep}$ is uniformly bounded (on $X$) in $\ep$;
\item[$\cdot$] When $\ep$ goes to $0$, $\ome$ converges on $\Xlc$ to $\om_{\D}$ having mixed Poincaré and cone singularities along $\D$.\\
\end{enumerate}

As $\ome$ is a Kähler metric on $\Xlc$ with Poincaré singularities along $\Dlc$, the case $J=\emptyset$ treated by Kobayashi (\cite{KobR}) and Tian-Yau (\cite{Tia}), cf section \ref{sec:log}, Theorem \ref{thm:kob},  enables us to find a smooth $\ome$-psh function $\vpe$ on $\Xlc$ satisfying:
\begin{equation}
\label{mae}
(\ome+\ddc \vpe)^{n}=e^{\vpe+F_{\ep}}\ome^{n}
\end{equation}
where 
\[F_{\ep}=f+\psie+\log \left( \frac{\om^{n}}{\prod_{j\in J}(|s_j|^{2}+\ep^{2})^{a_j}\ome^{n}} \right)\]
belongs to $\cka(\Xlc)$ thanks to Lemma \ref{lem:cka} and the assumptions on $f$. We may insist on the fact that the relation $\Fe \in \cka(\Xlc)$ is only qualitative in the sense that we \textit{a priori} don't have uniform estimates on $||\Fe||_{k,\alpha}$. 

\noindent
Besides, $\vpe\in\cka(\Xlc)$ (cf. \cite[section 3]{KobR}) so that in particular, it is bounded on $\Xlc$, $\ome+\ddc \vpe$ defines a complete Kähler metric on $\Xlc$, and the Ricci curvature of $\ome+\ddc \vpe$ bounded (from below) if and only if the one of $\ome$ is bounded (from below). Note that the bounds may \textit{a priori} not be uniform in $\ep$ - however we will show that this is the case. \\
Once observed that $\ome$ converges to a Kähler metric with mixed Poincaré and cone singularities along $\D$, and that equation \eqref{mae} is equivalent to 
\[ (\om+\ddc(\vpe+\psie))^{n}=\frac{e^{f+(\vpe+\psie)}}{\prod_{j\in J}(|s_j|^{2}+\ep^{2})^{a_j}}\om^n \]
the proof of our theorem boils down to showing that one can extract a subsequence of $(\vpe)_{\ep}$ converging to $\vp$, smooth outside $\D$, and such that $\om+\ddc \vp$ has the expected singularities along $\D$.

\subsection{\texorpdfstring{Establishing estimates for $\vpe$}{Establishing estimates for the potential}}
\label{subs:est}
In view of the \textit{a priori} estimates of section \ref{sec:estimates}, we first need to find a bound $\sup |\vpe| \le C$. We will see at the beginning of section \ref{subs:lb} that $\sup_{\ep} \sup_X |\Fe|$ is finite. Therefore, using \ref{cy1} with $\ome$ as reference metric, we have the desired $\mathscr C^0$ estimate: $\sup |\vpe| \le \sup_{\ep} \sup_X |\Fe|$.
So it remains to check that (here uniformly means "uniformly in $\ep$"):
\begin{enumerate}
\item[$(i)$] The bisectional curvature of $(\Xlc,\ome)$ is uniformly bounded from below;
\item[$(ii)$] $\Fe$ is uniformly bounded;
\item[$(iii)$] The Laplacian of $\Fe$ with respect to $\ome$, $\D_{\ome} \Fe$, is uniformly bounded.
\end{enumerate}
Once we will have shown that conditions $(i)-(iii)$ hold, we will get the existence of $C>0$ such that for all $\ep>0, \tr_{\ome}(\ome+\ddc \vpe) \le C$ (by the remarks above, $\ome+\ddc \vpe$ is complete and will have Ricci curvature bounded from below so that the assumptions of Proposition \ref{cy2} are fulfilled). Therefore, we will have $\ome+\ddc \vpe \le C \ome$. Furthermore, as $\vpe$ and $\Fe$ will be bounded, the identity $(\ome+\ddc \vpe)^{n}=e^{\vpe+F_{\ep}}\ome^{n}$ joint with the basic inequality $\det_{\ome}(\ome+\ddc \vpe) \cdot \tr_{\ome+\ddc \vpe} (\ome) \le (\tr_{\ome}(\ome+\ddc \vpe))^{n-1}$ (which amounts to saying that $\sum_{|I|=n-1}\prod_{i\in I} \lambda_i \le \left(\sum_{i=1}^n \lambda_i \right)^{n-1}$) will imply that, up to increasing $C$, $\tr_{\ome+\ddc \vpe} (\ome) \le C$. Therefore, 
\[C^{-1} \ome \le \ome+\ddc \vpe \le C \ome\]
and passing to the limit (after choosing a subsequence so that $(\vpe)_{\ep}$ converges to $\vp$ smooth outside $\Supp(\D)$ - we skip some important details here, cf. section \ref{subs:end}) our solution $\om_{o}+\ddc \vp$ will have mixed Poincaré and cone singularities along $\D$.\\

\subsubsection{A precise expression of the metric}
Before we go any further, we have to give the explicit local expresssions of $\ome$. We recall that $\D= \sum_{j\in J} a_j \D_j+\sum_{k\in K} \D_k$ for some disjoints sets $J,K\subset \mathbb N$, such that for all $j\in J$, $a_j<1$. In the following, an index $j$ (resp. $k$) will always be assumed to belong to $J$ (resp. $K$).

\noindent
First of all, pick some point $p_0\in X$ sitting on $\Supp(\D)$. We choose a neighborhood $U$ of $p_0$ trivializing $X$ and such that $\mathrm {Supp}(\D)\cap U= \{\prod_{J_U} z_j \cdot \prod_{K_U} z_k =0 \}$ for some $J_U\subset J$ and $K_U \subset K$. Then if $i\notin J_U \cup K_U$, $\D_i$ does not meet $U$. To simplify the notations, one may suppose that $J_U=\{1,\ldots,r\}$ and $K_U=\{r+1,\ldots, d\}$. Finally, we stress the point that although $p_0\in \Supp(\D)$, all our computations will be done on $U\cap \Xlc=U\setminus \Supp(\Dlc)$.\\

 So as to simplify the computations, we will use the following (more or less basic) lemma, extracted from \cite[Lemma 4.1]{CGP}:

\begin{lemm}
\phantomsection
\label{lem:coor}
Let $(L_1, h_1),\ldots, (L_d, h_d)$ be a set of hermitian
line bundles on a compact Kähler manifold $X$, and for each index $j= 1,\ldots,d$, let $s_j$ be a section of $L_j$; we assume that the hypersurfaces 
$$Y_j:= (s_j= 0)$$
are smooth, and that they have strictly normal intersections. 
Let $p_0\in \bigcap Y_j$; then there exist a constant $C>0 $ and an open set $V\subset X$ centered at $p_0$, such that for any point $p\in V$ there exists a coordinate system $z= (z_1,\ldots, z_n)$ at $p$ and a trivialization $\theta_j$ for $L_j$ such that:

\begin{enumerate}
\item [$(i)$] For $j= 1,\ldots, d$, we have $Y_j\cap V= (z_j= 0)$;

\item [$(ii)$] With respect to the trivialization $\theta_j$, the metric
$h_j$ has the weight $\varphi_j$, such that 
\begin{equation*}
\varphi_j(p)= 0, \quad d\varphi_j(p)= 0, \quad 
\left|\frac{\partial^{|\alpha|+|\beta|}\varphi_j}{\partial z_{\alpha}\partial \bar z_{\beta}}(p)\right|\le C_{\alpha, \beta}
\end{equation*}
for all multi indexes $\alpha, \beta$.
\end{enumerate}
\end{lemm}

Up to shrinking the neighborhood $V$, we may assume that each coordinate system $(z_1, \ldots , z_n)$ for $V$, as given in Lemma \ref{lem:coor}, satisfies $\sum_i |z_i|^{2} \le 1/2$. Moreover, in order to make the notations clearer, we define, for $i\in \{1,\ldots, n\}$, a non-negative function on $V$ (depending on $\ep$) by
\begin{equation*}
A(i)= 
\begin{cases} 
(|z_i|^{2}+\ep^{2})^{a_i/2}   & \text{if $i\in \{1,\ldots, r\}$;}\\
|z_i| \log \frac{1}{|z_i|^{2}} &\text{if $i\in \{r+1,\ldots, d\}$;}\\
1  &\text{if $i<d$.}\\
\end{cases}
\end{equation*}
Now, for $i,j,k,l \in \{1,\ldots, n\}$, we simply set $A(i,j,k,l):=A(i)A(j)A(k)A(l)$.\\

We first want to check that the holomorphic bisectional curvature of $\ome$ is bounded from below, that is 
\begin{equation}
\label{eq:cm}
\Theta_{\ome}(T_X) \ge -C \ome \otimes \mathrm{Id}_{T_X}
\end{equation}
for some $C>0$ independent of $\ep$, and where $\Theta_{\ome}(T_X)$ denotes the curvature tensor of the holomorphic tangent bundle of $(\Xlc,\ome$). It is useful for the following to reformulate the (intrinsic) condition \eqref{eq:cm} in terms of local coordinates.
Namely, the inequality in \eqref{eq:cm} amounts to saying that the following inequality holds:

\begin{equation}
\label{in2}
\sum_{p, q, r, s}R_{p\bar q r \bar s}(z)v_p\overline{v_q}w_r\overline{w_s}\ge - C|v|_{\ome}^2
|w|_{\ome}^2
\end{equation}
for any 
vector fields $\displaystyle v= \sum_p v_p\frac{\partial}{\partial z_p}$ and 
$\displaystyle w= \sum_r w_r\frac{\partial}{\partial z_r}$. 

\noindent The notation in the above relations is as follows:
in local coordinates, we write
\[\ome= \frac{i}{2}\sum_{p,q}g_{p\bar q} \, dz_p \wedge  d \bar z_q;\]
(so that the $g_{p\b q}$'s actually depend on $\ep$, but we choose not to let it appear in the notations so as to make them a bit lighter) and the corresponding components of the curvature tensor are
\[R_{ p\bar q  r \bar s }:= - \frac{\partial^2 g_{p\bar q}}{\partial z_r\partial \bar z_{s}}+
\sum_{k, l}g^{k\bar l} \frac{\partial g_{ p\bar k}}{\partial z_{r}} \frac{\partial g_{ l\bar q}}{\partial \bar z_{s}}
.\]

\noindent
Looking at the local expression of $\ome$ makes it clear that there exists $C>0$ independent of $\ep$ such that on $V$, $C^{-1} \om_{\D,\ep} \le \ome \le C\, \om_{\D,\ep}$, where
 \[\om_{\D,\ep}:=\sum_{j=1}^{r} \frac{i dz_j\wedge d\bar z_j}{(|z_j|^2+\ep^2)^{a_j}} + \sum_{k=r+1}^{d} \frac{i dz_k\wedge d\bar z_k}{|z_k|^{2} \log^{2}|z_k|^{2}}+\sum_{l=d}^{n} i dz_l\wedge d\bar z_l \] 
Therefore, if $\displaystyle v= \sum_p v_p\frac{\partial}{\partial z_p}$ satisfies $|v|_{\ome}=1$, then for each  $p$, $|v_p| \le A(p)$. We are now going to show the following two facts, which will ensure that the holomorphic bisectional curvature of $\ome$ is bounded from below:
\begin{enumerate}
\item[$(i)$] For every four-tuple $(p,q,r,s)$ with $\#\{p,q,r,s\}\ge 2$, we have $A(p,q,r,s) |R_{p\bar q r \bar s}(z)| \le C$;
\item[$(ii)$] For every $p$, and every $\ome$-unitary vector fields $v,w$, $|v_p|^{2}_{\ome} |w_p|^{2}_{\ome} R_{p\bar p p \bar p} \ge -C$.
\end{enumerate}

In order to prove $(i)-(ii)$, we have to give a precise expression of the metric $\ome$ in some coordinate chart. We will use the coordinates given by Lemma \ref{lem:coor}, which will simplify the computations a lot. We remind that $\ome=\om+\ddc \psie$, and according to \cite[equation (21)]{CGP} and Definition \ref{def:cg} (or \cite[pp. 50-51]{Gri}), the components $g_{p\bar q}$ of $\ome$ are given by:

\begin{eqnarray}
\phantomsection
\label{eqn:metric}
g_{p\bar q}&=& u_{p\bar q} +\frac{\delta_{pq,J}e^{-\vp_p}}{(|z_p|^{2}e^{-\vp_p}+\ep^{2})^{a_p}} +\delta_{p,J} e^{-\vp_p}\frac{\bar z_p \overline{\alpha_{qp}}}{(|z_p|^{2}e^{-\vp_p}+\ep^{2})^{a_p}} + \delta_{q,J}e^{-\vp_q}\frac{z_q \alpha_{qp}}{(|z_q|^{2}e^{-\vp_q}+\ep^{2})^{a_q}}  \nonumber \\
&+& \sum_{j\in J} \frac{|z_j|^{2}\beta_{jpq}}{(|z_j|^{2}e^{-\vp_j}+\ep^{2})^{a_j}}
\left((|z_j|^{2}e^{-\vp_j}+\ep^{2})^{1-a_j}-\ep^{2(1-a_j)}\right)\frac{\d^{2}\vp_j}{\d z_p \d \bar z_q}   \\
&+& \delta_{pq,K} \frac{i dz_p \wedge d \bar z_p}{|z_p|^{2}\log^{2}|z_p|^{2}}+\frac{\delta_{p,K} \lambda_p}{z_p \log^{2}|z_p|^{2}}+\frac{\delta_{q,K} \mu_q}{\bar z_q \log^{2}|z_q|^{2}}+ \sum^{d}_{k=r+1}\frac{\nu_k}{\log |z_k|^{2}} \nonumber
\end{eqnarray}
where $u_{pq},\alpha_{pq}, \beta_{jpq}, \lambda_p, \mu_q, \nu_k$ are smooth functions on $X$ (more precisely on the whole neighborhood $V$ of $p$ in $X$ given by Lemma \ref{lem:coor}). Moreover, $\alpha, \lambda, \mu$ (resp. $\beta$) are functions of the partial derivatives of the $\vp_i$'s; in particular, they vanish at the given point $p$ at order at least $1$ (resp. $2$). Finally, we use the notation $\delta_{p,J}=\delta_{p\in J}$ and $\delta_{pq,J}=\delta_{pq}\delta_{p\in J}$ (\textit{idem} for $K$ instead of $J$).

\subsubsection{Bounding the curvature from below}
First of all, using \eqref{eqn:metric}, and remembering that $\alpha, \beta, \lambda, \mu, $ vanish at $p$, on can give a precise $0$-order estimate on the metric (more precisely on the inverse matrix of the metric), which is a straightforward generalization of \cite[Lemma 4.2]{CGP}:

\begin{lemm}
\phantomsection
\label{lem:ordre0}
In our setting, and for $|z|^{2}+\ep^{2}$ sufficiently small, we have  at the previously chosen point $p$:
\begin{enumerate}
\item[$(i)$] For all $i\in \{1,\ldots,n\}$, \,  $g^{i\b i}= A(i)^{2}(1+\O(A(i)^{2}))$;
\item[$(ii)$] For all $j,k\in \{1,\ldots,n\}$ such that $j\neq k$, \, $g^{j\b k}= \O(A(j,k)^{2})$.
\end{enumerate}
\end{lemm}
 We insist on the fact that the $\O$ symbol refers to the expression $|z|^{2}+\ep^{2}=|z_1|^{2}+\cdots+|z_n|^{2}+\ep^{2}$ going to zero.\\

To bound the curvature, we will essentially have to deal with the Poincaré part of $\ome$, the other cone part being almost already treated in \cite{CGP}. We could use the fact that $(\Xlc, \om)$ has bounded geometry at any order (cf section \ref{sec:log}), but as mixed terms involving the (regularized) cone metric will appear \--- which is not known to be of bounded geometry\---, we prefer to give the explicit computations for more clarity.\\
For $\lambda$ and $\mu$ any smooth functions on $V$, there exist smooth functions $\lambda_1, \lambda_2, \ldots$ and $\mu_1, \mu_2, \ldots$ such that for any $k\in K$:
\begin{eqnarray*}
\frac{\d}{\d z_k} \left( \frac{\lambda}{z_k \log^{2} |z_k|^{2}} \right) &=& \frac{\lambda_1}{z_k \log^{2} |z_k|^{2}}+\frac{\lambda_2}{z_k^{2} \log^{2} |z_k|^{2}}+\frac{\lambda_3}{z_k^{2} \log^{3} |z_k|^{2}}=\O\left( \frac{1}{|z_k|^{2} \log^{2} |z_k|^{2}}\right) \\
\frac{\d}{\d \bar z_k} \left( \frac{\lambda}{z_k \log^{2} |z_k|^{2}}\right) &=& \frac{\lambda_4}{z_k \log^{2} |z_k|^{2}}+\frac{\lambda_5}{|z_k|^{2} \log^{3} |z_k|^{2}}=\O\left( \frac{1}{|z_k|^{2} \log^{3} |z_k|^{2}}\right) \\
\frac{\d^{2}}{\d z_k\d \bar z_k} \left( \frac{\lambda}{z_k \log^{2} |z_k|^{2}}\right) &=& \frac{\lambda_6}{z_k \log^{2} |z_k|^{2}}+\frac{\lambda_7}{|z_k|^{2} \log^{3} |z_k|^{2}}
+ \frac{\lambda_8}{z_k^{2} \log^{2} |z_k|^{2}} +\\
&+&\frac{\lambda_9}{z_k|z_k|^{2} \log^{3} |z_k|^{2}} + \frac{\lambda_{10}}{z_k^{2} \log^{3} |z_k|^{2}}+\frac{\lambda_{11}}{z_k|z_k|^{2} \log^{4} |z_k|^{2}}\\ 
&=& \O\left( \frac{1}{|z_k|^{3} \log^{3} |z_k|^{2}}\right) \\
\frac{\d}{\d z_k} \left(\frac{\mu}{\log |z_k|^{2}} \right) &=& \frac{\mu_1}{\log |z_k|^{2}}+\frac{\mu_2}{z_k \log^{2} |z_k|^{2}}=\O\left( \frac{1}{|z_k| \log^{2} |z_k|^{2}}\right) \\
\frac{\d}{\d \bar z_k} \left(\frac{\mu}{\log |z_k|^{2}} \right) &=& \frac{\mu_3}{\log |z_k|^{2}}+\frac{\mu_4}{\bar z_k \log^{2} |z_k|^{2}}=\O\left( \frac{1}{|z_k| \log^{2} |z_k|^{2}}\right) \\
\frac{\d^{2}}{\d z_k\d \bar z_k} \left(\frac{\mu}{\log |z_k|^{2}} \right) &=& \frac{\mu_5}{\log |z_k|^{2}}+\frac{\mu_6}{\bar z_k \log^{2} |z_k|^{2}}+\frac{\mu_7}{z_k \log^{2} |z_k|^{2}}+\frac{\mu_8}{|z_k|^{2} \log^{3} |z_k|^{2}}\\
&=&\O\left( \frac{1}{|z_k|^{2} \log^{3} |z_k|^{2}}\right) \\
\frac{\d}{\d z_k} \left( \frac{1}{|z_k|^{2} \log^{2} |z_k|^{2}}\right) &=& \frac{-1}{z_k |z_k|^{2} \log^{2} |z_k|^{2}}+\frac{-2}{z_k|z_k|^{2} \log^{3} |z_k|^{2}}=\O\left( \frac{1}{|z_k|^{3} \log^{2} |z_k|^{2}}\right) \\
\frac{\d^{2}}{\d z_k\d \bar z_k}  \left( \frac{1}{|z_k|^{2} \log^{2} |z_k|^{2}}\right) &=& \frac{1}{|z_k|^{4} \log^{2} |z_k|^{2}}+\frac{4}{|z_k|^{4} \log^{3} |z_k|^{2}}+\frac{6}{|z_k|^{4} \log^{4} |z_k|^{2}}\\
\end{eqnarray*}

As we are mostly interested in the Poincaré part of the metric $g$, we will write $g=g^{(P)}+g^{(C)}$ its decomposition into the Poincaré and the cone part (cf. the expression \eqref{eqn:metric}). Moreover, we write $g^{(P)}=\gamma^0+\gamma $ where $\gamma^0=\sum_{k\in K} \frac{i dz_k \wedge d \bar z_k}{|z_k|^{2} \log^{2} |z_k|^{2}}$. Therefore, if $k\neq l$, $g^{(P)}_{k\b l} = \gamma_{k\b l}$, and the computations above lead to (for every $k,l,r,s \in K$):
\begin{eqnarray}
\phantomsection
\frac{\d g^{(P)}_{k\b l}}{\d z_k} &=& \O \left( \frac{1}{A(k)^2A(l)} \right) \qquad \mathrm{if} \, \,  k\neq l \label{derp1} \\
\frac{\d^2 g^{(P)}_{k\b l}}{\d z_k \d \bar z_r} &=& \O \left( \frac{1}{A(k)^2A(r,l)} \right) \quad \, \, \mathrm{if} \, \,  k\neq l \label{derp2} \\ 
\frac{\d \gamma_{k\b l}}{\d z_r} &=& \O \left( \frac{1}{A(k,l,r)} \right)  \label{derp3} \\
\frac{\d^2 \gamma_{k\b l}}{\d z_r \d \bar z_s} &=& \O \left( \frac{1}{A(k,l,r,s)} \right)  \label{derp4}
\end{eqnarray}
Furthermore, we may note that if $\{p,q,r,s\}\cap J = \emptyset$, then we can see from the expression \eqref{eqn:metric} that $\frac{\d g_{p\b q}}{\d z_r}=\frac{\d g_{p\b q}^{(P)}}{\d z_r }+\O(1)$ as well as $\frac{\d^2 g_{p\b q}}{\d z_r \d\bar z_s}=\frac{\d g_{p\b q}^{(P)}}{\d z_r \d\bar z_s}+\O(1)$.
From this, \eqref{derp1}-\eqref{derp2} and Lemma \ref{lem:ordre0}, we deduce that for every $p,q,r,s\in K$ such that $p \neq q$, the expression $A(p,q,r,s)R_{p\b q r \b s}(z)$ is uniformly bounded in $z\in V\cap \Xlc$.\\

So it remains to study the terms of the form $R_{p\b p r \b s}$ for $p,r,s\in K$. And as mentionned in the last paragraph, the terms in the curvature tensor coming from the cone part (or the smooth part) do not play any role here, so we have:
\begin{eqnarray*}
R_{p\b p r \b s}&=&-\frac{\d^2 g_{p\b p}}{\d z_r \d\bar z_s} + \sum_{1\le k,l \le n}  g^{k\b l} \frac{\d g_{p\b l}}{\d z_r} \frac{\d g_{k\b p}}{\d \b z_s} \\
&=&-\frac{\d^{2}}{\d z_r\d \bar z_s}  \left( \frac{1}{|z_p|^{2} \log^{2} |z_p|^{2}}\right) - \frac{\d^{2} \gamma_{p\b p}}{\d z_r\d \bar z_s} + \sum_{1\le k,l \le n} g^{k\b l} \frac{\d g_{p\b l}^{(P)}}{\d z_r} \frac{\d g_{k\b p}^{(P)}}{\d \b z_s} + \O(1)
\end{eqnarray*}
Using \eqref{derp1}-\eqref{derp4} and Lemma \ref{lem:ordre0}, we see that the only possibly unbounded terms (when multiplied by $A(p)^2A(r,s)$) appearing in the expansion of $R_{p\b p r \b s}$ are coming from $\gamma_0$. More precisely, these are the following ones, appearing in $R_{p\b p p \b p}$ only:
\begin{equation}
\label{unbdd}
-\frac{\d^{2}}{\d z_p\d \bar z_p}  \left( \frac{1}{|z_p|^{2} \log^{2} |z_p|^{2}}\right) + \sum_{p\in \{k,l\}} g^{k\b l} \frac{\d g_{p\b l}^{(P)}}{\d z_p} \frac{\d g_{k\b p}^{(P)}}{\d \b z_p}
\end{equation}

Let us now expand the terms under the sum:

\begin{eqnarray}
\frac{\d g^{(P)}_{k \b p}}{\d z_p} &=& \O \left( \frac{1}{|z_p|^{2} \log^{3} |z_p|^{2}} \right) \qquad \mathrm{if} \, \,  k\neq p \label{unbdd2}  \\
\frac{\d g^{(P)}_{p \b p}}{\d z_p} &=& \frac{-1}{z_p |z_p|^{2} \log^{2} |z_p|^{2}}+\frac{-2}{z_p|z_p|^{2} \log^{3} |z_p|^{2}}+ \O\left(\frac{1}{|z_k|^2 \log^3 |z_k|^2} \right) \label{unbdd3}\\
\left|\frac{\d g_{p\b p}^{(P)}}{\d z_p} \right|^{2} &=& \frac{1}{|z_p|^{6} \log^{4} |z_p|^{2}} \left( 1+ \frac{4}{\log |z_k|^2} + \frac{4}{\log^{2} |z_k|^{2}} +\O(|z_k|)\right) \label{unbdd4} \\ \nonumber
\end{eqnarray}
Now, if we combine Lemma \ref{lem:ordre0} with \eqref{unbdd2}-\eqref{unbdd3}, we see that the remaining possibly unbounded terms (after multiplying by $A(p)^4$) appearing in \eqref{unbdd} are 

\[-\frac{\d^{2}}{\d z_p\d \bar z_p}  \left( \frac{1}{|z_p|^{2} \log^{2} |z_p|^{2}}\right) +  g^{p\b p} \frac{\d g_{p\b p}^{(P)}}{\d z_p} \frac{\d g_{p\b p}^{(P)}}{\d \b z_p}\]
which, thanks to point $(i)$ of Lemma \ref{lem:ordre0} and \eqref{unbdd4}, happens to be a $\O\left(\frac{1}{|z_p|^{4} \log^{4} |z_p|^{2} }\right)$, which finishes to prove that for every $p,q,r,s\in K$, the expression $A(p,q,r,s)R_{p\b q r \b s}(z)$ is uniformly bounded in $z\in V\cap \Xlc$.\\

Now we may look at the terms $R_{p\b q r \b s}$ where $p,q\in K$ but $r,s\notin K$. If $r,s\notin J$, then $A(p,q,r,s)R_{p\b q r \b s}(z)=A(p,q)R_{p\b q r \b s}(z)$ is uniformly bounded in $z\in V\cap \Xlc$ as we can see by looking at the expression of the metric \eqref{eqn:metric}. So now we may suppose that $r$ or $s$ belongs to $J$. The only term in the metric which may cause trouble is $ \sum_{j\in J} \frac{|z_j|^{2}\beta_{jpq}}{(|z_j|^{2}e^{-\vp_j}+\ep^{2})^{a_j}}+
\left((|z_j|^{2}e^{-\vp_j}+\ep^{2})^{1-a_j}-\ep^{2(1-a_j)}\right)\frac{\d^{2}\vp_j}{\d z_p \d \bar z_q}  $. But Lemma \ref{lem:ordre0} enables us to use the computations of \cite[section 4.3]{CGP} word for word, so as to show that $A(p,q,r,s)R_{p\b q r \b s}(z)$ is uniformly bounded in $z\in V\cap \Xlc$.\\

The next step in bounding the curvature of $\ome$ from below consists now in looking at the terms $R_{p\b q r \b s}$ for $p,q \in J$. Then the terms in $g_{p\b q}$ coming from the Poincaré part are of the form $\sum_k \frac{\nu_k}{\log |z_k|^{2}}$ as \eqref{eqn:metric} shows. These terms are uniformly bounded in $V\cap \Xlc$, as well as their  derivatives with respect to the variables $z_r, \bar z_s$ as long as $r,s \notin K$; in that that case \cite[sections 4.3-4.4]{CGP} gives us the expected lower bound for $A(p,q,r,s) R_{p\b q r \b s}$. If now $r\in K$, then we saw earlier that $A(r) \frac{\d}{\d z_r} \left(\frac{\nu_r}{\log |z_r|^{2}} \right)$, $A(s)\frac{\d}{\d \bar z_s} \left(\frac{\nu_s}{\log |z_k|^{2}} \right)$, $A(r)^2 \frac{\d^{2}}{\d z_r\d \bar z_r} \left(\frac{\nu_r}{\log |z_r|^{2}} \right)$ are bounded functions in $V\cap \Xlc$, so that, using Lemma \ref{lem:ordre0}, the boundedness of $A(p,q,r,s) R_{p\b q r \b s}$ is equivalent to the one of $A(p,q,r,s) R_{p\b q r \b s}^{g^{(C)}}$ whenever $p,q \in J$. And by \cite[section 4.3]{CGP}, we know the existence of this bound (which is an upper and lower bound, as $\#\{p,q,r,s\}\ge 2$) . \\

Finally, for the last step, we need to look at mixed terms $R_{p\b q r \b s}$ for $p \in K$ and $q \in J$ (or one of those not belonging to $J\cup K$). As $p\neq q$, the operators $A(r) \frac{\d}{\d z_r}, A(s) \frac{\d}{\d \bar z_s}$ and $A(r,s) \frac{\d^{2}}{\d z_r\d \bar z_r}$ map $g_{p\bar q}$ to a bounded function, as can be checked separately for $g^{(P)}$ (cf. the previous computations) and $g^{(C)}$ (cf. \cite[section 4.3]{CGP}). \\
\noindent
So we are done: $\ome$ has holomorphic bisectional curvature \textit{uniformly}  bounded from below on $\Xlc$.

\subsubsection{\texorpdfstring{Bounding the $\ome$-Laplacian of $F_{\ep}$}{Bounding the Laplacian}}
\label{subs:lb}
Remember that \[\Fe =f+\psie+\log \left( \frac{\om^{n}}{\prod_{j\in J}(|s_j|^{2}+\ep^{2})^{a_j}\ome^{n}} \right)\]
At the point $x$ (which is point $p$ of Lemma \ref{lem:coor}) , the $(p,\b q)$ component of $\ome(x)$ is
\begin{eqnarray*}
g_{p\bar q}(x)&=& u_{p\bar q}(x) +\frac{\delta_{pq,J}}{(|z_p|^{2}+\ep^{2})^{a_p}} +
\sum_{j\in J}  \left((|z_j|^{2}+\ep^{2})^{1-a_j}-\ep^{2(1-a_j)}\right)\frac{\d^{2}\vp_j}{\d z_p \d \bar z_q}(x)  + \\
&+& \delta_{pq,K} \frac{i dz_p \wedge d \bar z_p}{|z_p|^{2}\log^{2}|z_p|^{2}}+ \sum^{d}_{k=r+1}\frac{\nu_k}{\log |z_k|^{2}} 
\end{eqnarray*}
whereas the $(p,\b q)$ component of $\om(x)$ is
\[g_{p\bar q}^{(P)}(x)= u_{p\bar q}(x)+ \delta_{pq,K} \frac{i dz_p \wedge d \bar z_p}{|z_p|^{2}\log^{2}|z_p|^{2}}+ \sum^{d}_{k=r+1}\frac{\nu_k}{\log |z_k|^{2}}\]
Expanding the determinant of those metrics makes it clear that there exists $C>0$ such that 
\[ C^{-1} \le  \frac{\om^{n}}{\prod_{j\in J}(|s_j|^{2}+\ep^{2})^{a_j}\ome^{n}}  \le C\]
so that $\Fe$ is bounded on $\Xlc$.\\

Let us now get to bounding $\Delta_{\omega_{\varepsilon}} \Fe$. Actually we will show that $\pm \ddc \Fe \le C\ome$ for some uniform $C>0$, which is stronger than just bounding the $\ome$-Laplacian of $\Fe$, but we need this strengthened bound if we want to produce Kähler-Einstein metrics by resolving our Monge-Ampère equation.
\noindent
There are three terms in $\Fe$, namely $f$, $\psie$ and $\log \fe$ where \[\fe=\frac{\om^{n}}{\prod_{j\in J}(|s_j|^{2}+\ep^{2})^{a_j}\ome^{n}}\] The first two terms are easy to deal with: indeed, there exists $C>0$ (independent of $\ep$) such that $\ome \ge C^{-1} \om$ on $\Xlc$. Therefore, if one chooses $M$ such that $M\om\pm\ddc f>0$ (the assumptions on $f$ give the existence of such an $M$), then 
$ \ddc f \le CM \ome$. Moreover, $\ome = \om+\ddc \psie>0$  so that $\pm \ddc \psie \le \max(C,1) \ome$. Therefore it only remains to bound $\ddc \log \fe$ now.\\

We will use the following basic identities, holding for any smooth functions $f>0$ and $u,v$ on some open subset of $U\subset X$:
\begin{eqnarray}
\ddc \log f &=& \frac{1}{f} \ddc f + \frac{1}{f^2} df \wedge d^c f  \label{lap1} \\
\ddc \left( \frac 1 f \right)  &=& -\frac{1}{f^2} \ddc f + \frac{2}{f^3} df \wedge d^c f \label{lap2} \\
\ddc (uv)& =& u\, \ddc v + v \,\ddc u +  du \wedge d^c v -d^c u \wedge dv  \label{lap3} \\ 
\nabla (uv) &=& (\nabla u) \, v + u \, (\nabla v) \label{lap4}
\end{eqnarray}

We just saw that $\fe$ is bounded below by some fixed constant $C^{-1}>0$ on $\Xlc$, so that by \eqref{lap1}, $\pm \ddc \log \fe$ will be dominated by some fixed multiple of $\ome$ if we show that both $\pm \ddc \fe \le C \ome$ and $\left|\nabla_{\ep} \fe \right|_{\om} \le C$ for some uniform $C>0$ (the last term denotes the norm computed with respect to $\om$ of the $\ome$-gradient of $\fe$, defined as usual by $d\fe(X)=\ome(\nabla_{\ep}\fe, X)$ for every vector field $X$). For convenience, we will split the computation by writing
\begin{equation}
\label{eqnfe}
\fe =\left(\prod_{j\in J}(|s_j|^{2}+\ep^{2})^{a_j} \cdot  \prod_{k\in K} |s_k|^2\log^2|s_k|^2 \cdot \ome^{n} \right)^{-1} \cdot \left(\prod_{k\in K} |s_k|^2\log^2|s_k|^2 \cdot \om^{n} \right)
\end{equation}
By \eqref{lap2}-\eqref{lap3}, we only need to check that the gradient $\nabla_{\ep}$ of the terms inside the parenthesis is bounded, and that their $\pm \ddc$ is dominated by some fixed multiple of $\ome$. Let us begin with the second one, which is simpler:

\begin{lemm}
\phantomsection
\label{lem:calc}
Let $\om$ be a Kähler form of Carlson-Griffiths type on $\Xlc$, and let $\om_0$ be some smooth Kähler form on $X$. We set \[V=\left(\prod_{k\in K} |s_k|^2\log^2|s_k|^2\right) \cdot \frac{\om^{n}}{\om_0^n}\]
Then there exists $C>0$ such that $\pm V$ is $C\om$-psh on $\Xlc$.
\end{lemm}

\begin{proof}
We write, with our usual coordinates (cf Lemma \ref{lem:coor}) :
\begin{eqnarray}
\om^n &=& \prod_{k\in K} \frac{1}{|z_k|^2\log^2|z_k|^2} \left(1+ \sum_{K_i \subset K} A_i \prod_{k_i \in K_i} \frac{1}{\log |z_{k_i}|^2}\right) \label{det} \\
&+& \sum _{K_j,K_l,K_m,K_p\subset K }  A_{jlmp}\prod_{k_j\in K_j} \frac{1}{ |z_{k_j}|^2 \log^2 |z_{k_j}|^2} \cdot  \prod_{k_l\in K_l} \frac{1}{ z_{k_l} \log^2 |z_{k_l}|^2} \cdot \,\,  \nonumber \\ 
& & \cdot  \prod_{k_m\in K_m} \frac{1}{ \bar z_{k_m} \log^2 |z_{k_m}|^2} \cdot \prod_{k_p \in K_p}\frac{1}{\log |z_{k_p}|^2} \,\,\,\cdotp  \,   \Omega \nonumber
\end{eqnarray}
for $\Omega$ some smooth volume form on $X$ and where the second sum is taken over the subsets $K_j,K_l,K_m,K_p$ of $K$ that are disjoint, and where $A_i, A_{jlmp}$ are smooth functions on the whole $X$. Let us apply the operators $A(i,j) \frac{\d}{\d z_i \d \bar z_j}$ and $g^{i\bar j} \frac{\d}{\d z_i} \cdot \frac{\d} {\d \bar z_j}$ to $ \frac{1}{\log |z_k|^2}, z_k, \bar z_k,|z_k|^2 \log |z_k|^2$ and $|z_k|^2 \log^2 |z_k|^2$, and check that we obtain bounded functions. We already did it for the first term, so we only have to compute:
\begin{eqnarray*}
\frac{\d}{\d z_k}(|z_k|^2 \log |z_k|^2) &=& \bar z_k \log |z_k|^2 +\bar z_k =\O(1) \\
\frac{\d^2}{\d z_k \d \bar z_k}(|z_k|^2 \log |z_k|^2) &=& \log |z_k|^2+2=\O\left( \frac{1}{|z_k|^2 \log^2 |z_k|^2}\right)\\
\frac{\d}{\d z_k}(|z_k|^2 \log^2 |z_k|^2) &=& \bar z_k \log^2 |z_k|^2 +2\bar z_k \log |z_k|^2=\O(1) \\
\frac{\d^2}{\d z_k \d \bar z_k}(|z_k|^2 \log^2 |z_k|^2) &=& \log^2 |z_k|^2+4\log |z_k|^2 +2=\O\left( \frac{1}{|z_k|^2 \log^2 |z_k|^2}\right)\\
\end{eqnarray*}
This shows that the $\ome$-gradient of these factors (denote them generically $\kappa$) is bounded. As for $\ddc \kappa $, the previous computations show that in coordinates, its $(i,j)$-th term is uniformly bounded by $C A(i,j)$ for every $i,j$ (this is actually stronger than saying that it becomes bounded when multiplied with $g^{i\b j}$, condition which would however be sufficient to show that the $\ome$-Laplacian is bounded). Therefore, as the matrix of $\ome$ can be written $\mathrm{diag}(A(1)^2, \ldots, A(n)^2)+\O(1)$ in coordinates, and using the Cauchy-Schwarz inequality, one easily obtains $C>0$ such that $\pm \ddc \kappa \le C \ome$. \\
In fact, once we we saw that the only singular terms were $ \frac{1}{\log |z_k|^2},|z_k|^2 \log |z_k|^2$ and $|z_k|^2 \log^2 |z_k|^2$, we could have used the usual quasi-coordinates as in \ref{lem:cka} to conclude.
\end{proof}

Let us now get to the term inside the first parenthesis of \eqref{eqnfe}. For this, notice that in the expansion of $\ome^n$, we find the terms of \eqref{det} multiplied by terms of the form
\[C(z)+\sum_{I\subsetneq J} A_I(z) \prod_{i\in I} (|z_i|^2e^{-\vp_i}+\ep^2)^{a_i}\]
where $C(z)$ and $A_I(z)$ are sums of terms of the form
\begin{eqnarray*}
&& B(z) \prod_{j_l\in J_l}[(|z_{j_l}|^{2}e^{-\vp_{j_l}}+\ep^{2})^{1-a_{j_l}}-\ep^{2(1-a_{j_l})}] \cdot \prod_{j\in J_k} \frac{z_{j_k} \alpha_{j_k}}{(|z_{j_k}|^{2}e^{-\vp_{j_k}}+\ep^{2})^{\lambda_{j_k} a_{j_k}}} \cdot \,\, \cdots\\
&\cdots & \,\, \cdot \prod_{j\in J_m} \frac{\bar z_{j_m} \bar \alpha_{j_m}}{(|z_{j_m}|^{2}e^{-\vp_{j_m}}+\ep^{2})^{\lambda_{j_m} a_{j_m}}}
\prod_{j_p\in J_p} \frac{|z_{j_p}|^{2}\beta_{j_p}}{(|z_{j_p}|^{2}e^{-\vp_{j_p}}+\ep^{2})^{a_{j_p}}}   
\end{eqnarray*}
where $I,J_l,J_k, J_m,J_p$ are disjoint subsets of $J$, and where $B(z)$ is smooth independent of $\ep$, $\alpha_j$ is smooth and vanishes at $x$, $\beta_j$ is smooth and vanishes at order at least $2$ at $p$, and $\lambda_j\in \{0, 1/2\}$. And now, using Lemma \ref{lem:ordre0} and \cite[section 4.5]{CGP} (we must slightly change the argument therein as said above to control the $\ddc$ with respect to $\ome$ and not only $\De$), we can conclude that the appropriate $\ddc$ (resp. gradients) of those quantities are dominated by $C\ome$ (resp. bounded). Combining this with the previous computations, we deduce that $\De \Fe$ is bounded on the whole $\Xlc$.

\subsection{End of the proof}
\label{subs:end}
Remember that we wish to extract from the sequence of smooth metrics $\ome + \ddc \vpe$ on $\Xlc$ some subsequence converging to a smooth metric on $X\setminus \Supp(\D)$.
In order to do this, we need to have \textit{a priori} $\mathscr C^k$ estimates for all $k$. The usual bootstrapping argument for the Monge-Ampère equation allows us to deduce those estimates from the $\mathscr C^{2,\alpha}$ ones for some $\alpha\in ]0,1[$. The crucial fact here is that we have at our disposal the following \textit{local} result, taken from \cite{Gilb} (see also \cite{Siu}, \cite[Theorem 5.1]{Bl}), which gives interior estimates. It is a consequence of Evans-Krylov's theory: 
\begin{theo}
Let $u$ be a smooth psh function in an open set $\Omega \subset \CC^n$ such that $f:= \det(u_{i \bar j})>0$. Then for any $\Omega' \Subset \Omega$, there exists $\alpha \in]0, 1[$ depending only on $n$ and on upper bounds for $||u||_{\mathscr C^{0}(\Omega)}$, $\sup_{\Omega} \, \Delta \vp, ||f||_{\mathscr C^{0, 1}(\Omega)}, 1/\inf_{\Omega} \, f$, and $C>0$ depending in addition on a lower bound for $d(\Omega', \d \Omega)$ such that: 
\[||u||_{\mathscr C^{2, \alpha}(\Omega')}\le C.\]
\end{theo}

In our case, we choose some point $p$ outside the support of the divisor $\D$, and consider two coordinate open sets $\Omega' \subset \Omega$ containing $p$, but not intersecting $\mathrm{Supp}(\D)$. In that case, we may find a smooth Kähler metric $\omega_p$ on $\Omega$ such that on $\Omega'$, the covariant derivatives at any order of $\ome$ are uniformly bounded (in $\ep$) with respect to $\om_p$. Then one may take $u=\vpe$ in the previous theorem, and one can easily check that there are common upper bounds (i.e. independent of $\ep$) for all the quantities involved in the statement. This finishes to show the existence of uniform \textit{a priori} $\mathscr C^{2, \alpha}(\Omega')$ estimates for $\vpe$.\\

As we mentioned earlier, the ellipticity of the Monge-Ampère operator automatically gives us local \textit{a priori} $\mathscr C^k$ estimates for $\vpe$, which ends to provide a smooth function $\vp$ on $X\setminus \Supp(\D)$ (extracted from the sequence $(\vpe)_{\ep}$) such that $\omi=\om+\ddc \vp$ defines a smooth metric outside $\mathrm{Supp}(\D)$ satisfying 
\[ (\om+\ddc \vp)^{n} =\frac{e^{\vp+f}}{\prod_{j\in J}|s_j|^{2a_j}}\,\om^{n}. \]
Moreover, the strategy explained at the beginning of the previous section \ref{subs:est} and set up all along the section shows that this metric $\vp$ has mixed Poincaré and cone singularities along $\D$, so this finishes the proof of the main theorem.

\subsection{Remarks}

It could also be interesting to study the following equation: 
\[ (\om+\ddc \vp)^{n} =\frac{e^{f}}{\prod_{j\in J}|s_j|^{2a_j}}\,\om^{n} \]
where $\om$ is of Carlson-Griffith's type, and asked whether its eventual solutions have mixed Poincaré and cone singularities. This equation has been recently studied and solved by H. Auvray in \cite[Theorem 4]{Auvray} in the case where $\Dklt=0$ (the "logarithmic case"), and for $f$ vanishing at some order along $\D$. To adapt his results, one would need to show that one can make a choice of $\psie$ so that $\Fe$ vanishes along $\Dlc$ at some fixed order, what we have been unable to do so far.\\
However, adapting some recent results of pluripotential theory, we are able to prove the existence (and uniqueness) of solutions  to following equation:
\[(\om_0+\ddc \vp)^{n} =\frac{e^{f}}{\prod_{r\in J\cup K}|s_{r}|^{2a_{r}}}\,\om_0^{n}\]
where $\om_0$ is a Kähler form on $X$. We can't say much about the regularity of $\vp$; so far we only know that $\vp\in \mathcal{E}^1(X,\om_0)$.

\section{A vanishing theorem for holomorphic tensor fields}
Given a pair $(X,\D)$, where $X$ is a compact Kähler manifold and $\D=\sum a_i \D_i$ a $\R$-divisor with simple normal crossing support such that $0\le a_i \le 1$, there are many natural ways to construct holomorphic tensors attached to $(X,\D)$. 

\noindent
To begin with, one defines the tensor fields on a manifold $M$, which are contravariant of degree $r$ and covariant of degree $s$
as follows
\begin{equation}
\label{11}
T^r_sM:= \left(\otimes ^rT_M\right)\otimes \left(\otimes ^sT_M^\star\right).
\end{equation}
In our present context, we consider $M:= X_0$, that is to say the
Zariski open set $X\setminus \Supp(\D)$. Let us recall the definition of the \textit{orbifold tensors} introduced by F. Campana \cite{Camp2}. To avoid a possible confusion with the standard orbifold situation (\textit{ie} when $a_i=1-\frac 1 m$ for some integer $m$), we will not use his terminology and refer to these tensors as \textit{$\D$-holomorphic tensors}. \\

Let $x\in X$ be a point; since the hypersurfaces $(\D_i)$ have strictly normal intersections, there exist a small open set $\Omega\subset X$, together with a coordinate system $z= (z_1,\ldots, z_n)$ 
centered at $x$ such that $\D_i \cap \Omega= (z_i= 0)$ for $i= 1,\ldots, d$ and 
$\D_i\cap \Omega= \emptyset$ for the others indexes. 
We define the locally free sheaf $T^r_s(X|\D)$ generated as an $\Ox$-module by the tensors 
$$z^{\lceil (h_I- h_J)\cdot a\rceil}\frac{\partial} {\partial z_{ I}}\otimes dz^J$$
where the notations are as follows:

\begin{enumerate}

\item $I$ (resp. $J$) is a collection of positive integers in $\{1,\ldots, n\}$ of cardinal 
$r$ (resp. $s$) (we notice that we may have repetitions among the 
elements of $I$ and $J$, and we count each element according to its multiplicity).

\item For each $1\le i\le n$, we denote by 
$h_I(i)$ the multiplicity of $i$ as element of the collection $I$.

\item For each $i= 1,\ldots, d$ we have $a_i:= 1-\tau_i$, and
$a_i= 0$ for $i\ge d+ 1$.

\item We have 
$$z^{\lceil (h_I- h_J)\cdot a\rceil}:= \prod_i {(z^i)}^{\lceil (h_I(i)- h_J(i))\cdot 
a_i\rceil}$$

\item If $I= (i_1,\ldots, i_r)$, then we have
$${\partial\over \partial z_{I}}:= {\partial\over \partial z_{i_1}}\otimes \cdots
\otimes {\partial\over \partial z_{i_r}}$$
and we use similar notations for $dz^J$.

\end{enumerate}

\noindent Hence the holomorphic tensors we are considering here have
prescribed zeros/poles near $X\setminus X_0$, according to the 
multiplicities of $\D$. In the cone case ($\Dlc=0$), those tensors have a nice interpretation (\cite[Lemma 8.2]{CGP}):

\begin{lemm}
\phantomsection
\label{lem32} 
Assume $\Dlc=0$, and let $u$ be a smooth section of the bundle $T^r_s(X_0)$. Then 
$u$ corresponds to a holomorphic section of $T^r_s(X|D)$ if and only if 
$\bar \partial u= 0$ and $u$ is bounded with respect to some metric with cone singularities along $\D$. 
\end{lemm}

In \cite{CGP}, the vanishing and parallelism theorems are proved using the classical Bochner formula with an appropriate cut-off function for the space of bounded (for the cone metric) holomorphic sections of $T_s^r(X_0)$, and the lemma above enables to transfer this property to $\D$-holomorphic tensors. 

\noindent
Unfortunately, there is no such simple correspondence in the general log-canonical case. For example, if $\D$ has only one component (with coefficient $1$) of local equation $z=0$, then 
$ \frac{dz}{z}$ is a local section of $T_1^0(X|\D)$ but it is not bounded with respect to any metric having Poincaré singularities along $\D$. 

\noindent
The idea is to force $\D$-holomorphic tensors to be bounded by twisting them with the trivial line bundle $L=\Ox$ equipped with the singular hermitian metric 
\[h_L=e^{-2s \sum_k \log \log \frac{1}{|s_k|^2}}=  \prod_{k\in K} \frac{1}{\log^{2s} |s_k|^2}\]
where the $(s_k)_{k\in K}$ are the sections of the divisors $\D_k$ appearing in $\Dlc=\lceil \D \rceil$. In more elementary terms, we just change the reference metric measuring those tensors. Then, using a twisted Bochner formula, we will be able to carry on the computations done in \cite{CGP} to obtain the vanishing. It will be practical for the following to introduce the following notation:

\begin{defi}
\phantomsection
\label{def:bdd}
Let $(X,\D)$ be a pair such that $\D$ has simple normal crossing support and coefficients in $[0,1]$. The space of bounded holomorphic tensors of type $(r,s)$ for $(X,\D)$ is defined by
\[\mathscr H^{r,s}_B(X|\D)=\{u \in \mathscr C^{\infty}(X_0, T_s^r (X_0))\, ; \,  \exists C; \,  |u|_h^2 \le C \, \, \mathrm{and}\, \,  \bar \partial u=0 \}\]
where $h=g_{r,s}\otimes h_L$ is a metric on $T^r_s(X_0)$ induced by $h_L$ and a metric $g$ on $X_0$ having mixed Poincaré and cone singularities along $\D$.
\end{defi}
Of course, this definition does not depend on the choice of the metric $g$ having Poincaré and cone singularities along $\D$; it coincides with the one introduced in \cite{CGP} for klt pairs. The main point about this definition, which legitimates it, consists in the following proposition giving the expected identification between bounded and  $\D$-holomorphic tensors: 

\begin{prop}
\phantomsection
\label{prop:bdo}
With the previous notations, we have a natural identification: 
\[\mathscr H ^{r,s}_B(X|\D) = H^0(X, T_s^r(X|\D)).\]
\end{prop}

\begin{proof}
We only need to check it locally on $\Omega = ({\mathbb{D}^*})^p\times ({\mathbb{D}^*})^q\times \mathbb{D}^{n-(k+l)}$, where the boundary divisor restricted to $\Omega$ is given by $\sum_{k=1}^p d_k [z_i=0] + \sum_{k=p+1}^{p+q} [z_k=0] $, and we choose $g$ to be the model metric $\om_{\D}$ given in the introduction. \\

Let us begin with the inclusion $\mathscr H ^{r,s}_B(X|\D) \subset H^0(X, T_s^r(X|\D))$. By orthogonality of the different $\frac{\partial} {\partial z_{I}}\otimes dz^J$, we only have to consider $u=v\frac{\partial} {\partial z_{ I}}\otimes dz^J$ for some (holomorphic) function $v$ satisfying:
\[ \frac{|v|}{\prod_{k=1}^p |z_k|^{(h_I(k)-h_J(k))a_k}\prod_{k=p+1}^{p+q} |z_k|^{h_I(k)-h_J(k)} \left(\log \frac 1 {|z_k|^2}\right)^{s+h_I(k)-h_J(k)}} \le C \]
Consider now the function
\[ w:=\frac{v}{\prod_{k=1}^p z_k^{\lceil (h_I(k)-h_J(k))a_k\rceil }\prod_{k=p+1}^{p+q} z_k^{h_I(k)-h_J(k)}}   \]
whose modulus $|w|$ can also be rewritten in the form
\[  \frac{|v|}{\prod_{k=1}^p |z_k|^{(h_I(k)-h_J(k))a_k}\prod_{k=p+1}^{p+q} |z_k|^{h_I(k)-h_J(k)} \left(\log \frac 1 {|z_k|^2}\right)^{s+h_I(k)-h_J(k)}} \cdot\]\[ \frac{\prod_{k=p+1}^{p+q} \left(\log \frac 1 {|z_k|^2}\right)^{s+h_I(k)-h_J(k)}}{\prod_{k=1}^p |z_k|^{\lceil (h_I(k)-h_J(k))a_k\rceil - (h_I(k)-h_J(k))a_k }} \]
The first factor is bounded; moreover, using the fact that $0\le\lceil x \rceil -x<1$ for every real number $x$ and that $\left(\log \frac 1 {|z|}\right) ^{\alpha}$ is integrable at $0$ for every real number $\alpha$, we conclude that the second factor is also $L^2$. This finishes to prove that $w$ is $L^2$, so in particular it extends across the support of our divisor, and therefore, $u\in H^0(\Omega, T_s^r(\Omega|\D_{|\Omega}))$.\\

For the reverse inclusion, every "irreducible" $\D$-holomorphic tensor $u\in H^0(\Omega, T_s^r(\Omega|\D_{|\Omega}))$ can be written 
\[u=\prod_{k=1}^p z_k^{\lceil (h_I(k)-h_J(k))a_k\rceil }\prod_{k=p+1}^{p+q} z^{h_I(k)-h_J(k)} v \frac{\partial} {\partial z_{I}}\otimes dz^J \]
for some holomorphic function $v$, and some $I \in \{1, \ldots, n\}^r$, $J \in \{1, \ldots, n\}^s$. So for $g$ the metric on $X_0$ attached to $\om_{\D}$, and setting $h=g_{r,s}\otimes h_L$ as in Definition \ref{def:bdd}, we have:
\[ |u|_h  = \frac{ |v| \prod_{k=1}^p |z_k|^{\lceil (h_I(k)-h_J(k)) a_k\rceil-(h_I(k)-h_J(k))a_k }}{\prod_{k=p+1}^{p+q}  \left(\log \frac 1 {|z_k|^2}\right)^{s+h_I(k)-h_J(k)}}\]
which is clearly bounded near the divisor since $s+h_I(k)-h_J(k)\ge 0$ for all $k$.
\end{proof}

Now we can state the main result of this section, which is a partial generalization of \cite[Theorem C]{CGP}:

\begin{theo}
\phantomsection
\label{thm:bt}
Let $(X,\D)$ be a pair such that $\D=\sum a_i \D_i$ has simple normal crossing support, with coefficients satisfying:
$ 1/2 \le a_i \le 1$ for all $i$. \\
If $K_X+\D$ is ample, then there is no non-zero  $\D$-holomorphic tensor of type $(r,s)$ whenever $r\ge s+1$:
\[H^0(X, T_s^r(X|\D))=0.\]
\end{theo}

\begin{proof}[Proof of Theorem \ref{thm:bt}]
Proposition \ref{prop:bdo} allows us to reduce the vanishing of the $\D$-holomorphic tensors to the one of bounded tensors as defined in \ref{def:bdd}.
The proof of this result is similar to the one of \cite[Theorem C]{CGP}, the two main new features being the existence of a Kähler-Einstein metric with mixed Poincaré and cone singularities along $\D$ (cf. Theorem A), and the use of a twisted Bochner formula.
For this reason, we will give a relatively sketchy proof, and we will refer to \cite{CGP} for the details we skip.\\

\noindent
To fix the notations, we write $\D=\sum_{j\in J} a_j \D_j + \sum_{k\in k} \D_k$ where for all $j\in J$, we have $a_j<1$. In the following, any index $j$ (resp. $k$) will be implicitely assumed to belong to $J$ (resp. $K$), whereas the index $i$ will vary in $J\cup K$.

\noindent
As $K_X+\D$ is ample, Theorem A guarantees the existence of a Kähler metric $\omi$ on $X_0$ such that $-\Ric \omi = \omi$, and having mixed Poincaré and cone singularities along $\D$.  We choose now an element $u\in \mathscr H ^{r,s}_B(X|\D)$ with $r\ge s+1$, and we want to use a Bochner formula to show that $u=0$. 

\noindent
To do this, we need to perform a cut-off procedure, and control the error term so that one can pass to the limit in the cut-off process. Let us now get a bit more into the details. \\\\

\noindent
\textbf{Step 1: The cut-off procedure}\\
We define $\rho: X\to]-\infty, +\infty]$ by the formula
\[\rho(x):= \log\left(\log {1\over \prod_i |s_i(x)|^2}\right).\]
For each $\varepsilon> 0$, let $\displaystyle \chi_{\varepsilon}: [0,+ \infty[\to [0, 1]$ be a smooth function which is 
equal to zero on the interval $\displaystyle [0, {1/\varepsilon}]$, and which is equal to 1
on the interval $\displaystyle [1+ {1/\varepsilon}, +\infty]$. One may for example define $\chi_{\ep}(x)=\chi_{1}(x-\frac{1}{\ep})$, so that
$$\sup_{\varepsilon> 0, t\in \mathbb R_+}|\chi_{\varepsilon}^\prime(t)|\le C < \infty,$$
and we define $\theta_\varepsilon: X\to [0, 1]$ by the expression
$$\theta_\varepsilon(x)= 1- \chi_{\varepsilon}\big(\rho(x)\big).$$
\noindent We assume from the beginning that we have 
$$\prod_i|s_i|^2\le e^{-2}$$
at each point of $X$, and then it is clear that we have 
$$\theta_\varepsilon =1 \iff \prod_i|s_i|^2\ge e^{-e^{1/\varepsilon}}$$
and also
$$\theta_\varepsilon =0 \iff \prod_i|s_i|^2\le e^{-e^{1+ 1/\varepsilon}}.$$
We evaluate next the norm of the $(0, 1)$--form $\bar \partial \theta_\varepsilon$; we have
$$\bar \partial \theta_\varepsilon(x)=  \chi_{\varepsilon}^\prime\big(\rho(x)\big)
{1\over \log {1\over \prod_i |s_i(x)|^2}}\sum_i{\langle s_i, D^\prime s_i\rangle
\over |s_i|^2}(x).$$
As $\omi$ has mixed Poincaré and cone singularities along $\D$, we have:
\begin{equation}
\label{27}
|\bar \partial \theta_\varepsilon|_{\omega_\infty}^2\le 
{C|\chi_{\varepsilon}'(\rho)|^2\over \log^2 {1\over \prod_j |s_j|^2}}\left(\sum_j{1\over |s_j|^{2(1-a_j)}} +\sum_k \log^2 |s_k|^2\right)
\end{equation} 
at each point of $X_0$. Indeed, this is a consequence of the fact that the norm of the $(1, 1)$-forms
\[\frac{i\langle D^\prime s_j, D^\prime s_j\rangle}{ |s_j|^{2a_j}} \qquad \mathrm{and} \qquad  \frac{i\langle D^\prime s_k, D^\prime s_k\rangle}{ |s_k|^{2}\log^2 |s_k|^2} \]
with respect to $\omega_\infty$ are bounded from above by a constant.\\

\noindent
Let $\varepsilon> 0$ be a real number; we consider the tensor
$$u_\varepsilon:= \theta_\varepsilon u.$$
It has compact support, hence by the (twisted) Bochner formula (see e.g. \cite[Lemma 14.2]{Dem95}), we infer
\begin{equation}
\label{28}
\int_{X_0}|\overline\partial(\# u_\varepsilon)|^2_h dV_{\omega_\infty}= \int_{X_0}|\overline\partial u_\varepsilon |^2_h dV_{\omega_\infty}
+ \int_{X_0}\left( \langle \mathcal R(u_\varepsilon), u_\varepsilon\rangle_h + \gamma |u_{\ep}|_h^2\right)  dV_{\omega_\infty}
\end{equation}
where:
\begin{enumerate}
\item[$\cdot$] $\mathcal R$ is a zero-order operator such that in our case ($-\Ric \omi = \omi$), we have
$$R_{j\overline i}= -\delta_{ji},$$
and therefore the linear term $\displaystyle \langle \mathcal R(u_\varepsilon), u_\varepsilon\rangle$ becomes simply $(s-r)|u_\varepsilon |^2$;\\
\item[$\cdot$] $h=\om_{\infty,*}\otimes h_L$, where $\om_{\infty,*}$ denotes the canonical extension of $\omi$ to the appropriate tensor fields (which are respectively $T^s_r(X_0)\otimes \Omega^{0,1}(X_0)$, $T^r_s(X_0)\otimes \Omega^{0,1}(X_0)$ and $T^r_s(X_0)$);
\item[$\cdot$] $\gamma=\tr_{\omi}(\Theta_h(L))$ is the trace with respect to $\omi$ of the curvature of $(L,h)$.\\
\end{enumerate}

Here we need to be cautious because of the singularities of the metric $h_L$ on $\D$. Indeed, the Bochner formula applies to smooth hermitian metrics; howwever one can consider here some metric $h_{L,\ep}$ which would coincide with $h_L$ whenever $\theta_{\ep}>0$ and which is a smooth metric near $\D$. For example, on can set $h_{L,\ep}=\theta_{\ep/2}h_L+(1-\theta_{\ep/2})$. Then for each $\ep<1$, there exists an open set $U_{\ep} \supset \overline{\{\theta_{\ep}>0\}}$ on which $h_{L,\ep}=h_{L}$ so that in particular, in the formula \eqref{28}, one can replace $h_L$ by $h_{L,\ep}$ without affecting anything.\\

There remains two steps to achieve now: the first one consists in evaluating the correction term $\gamma$ induced by the curvature of $L$, and the second one is to show that the integration by part is valid in the  Poincaré-cone setting; more precisely we have to prove that the error term $\int_{X_0} |\bar \partial u_{\ep} |^2_{h} dV_{\omi}$ converges to $0$ as $\ep$ goes to $0$.\\\\

\noindent 
\textbf{Step 2: Dealing with the curvature of $(L,h)$}\\
We work on local charts where $\Dlc$ is given by $\{\prod_{k\in K} z_k =0\}$. 

\noindent
To begin with, we know that there exists $A>0$ such that  $\omi \le  A \left(\om_{\rm klt}+ \sum_k \frac{i dz_k \wedge d \bar z_k}{|z_k|^2\log^2 |z_k| ^2}\right)$ where $\om_{\rm klt}$ is some smooth metric on $X\setminus \Supp(\Dklt)$ having cone singularities along $\Dklt$. It will be useful to introduce the notation $\om_{\rm lc}:=\om_{\rm klt}+ \sum_k \frac{i dz_k \wedge d \bar z_k}{|z_k|^2\log^2 |z_k| ^2}$. Moreover, the usual computations (see e.g.\cite[Lemma 1]{KobR}) show that there exists a smooth $(1,1)$-form $\alpha$ on our chart satisfying
\[-\sum_{k\in K} \ddc \log \log \frac{1}{|s_k|^2} \ge \sum_{k\in K} \frac{i dz_k \wedge d \bar z_k}{|z_k|^2\log^2 |z_k| ^2} + \frac{1}{B}\alpha \]
where $B$ is a constant which can be taken as large as wanted up to scaling the (smooth) metrics on the $\D_k$'s, which does not affect their curvature.
Therefore, the curvature $\Theta_{h_L}(L)$ of $L$ satisfies:
\begin{eqnarray*}
\tr_{\omi} (-\Theta_{h_L}(L)) & \ge & A^{-1} \tr_{\om_{\rm lc}}(-\Theta_{h_L}(L))\\
&\ge & 2sA^{-1} \tr_{\om_{\rm lc}}\left(\sum_{k\in K} \frac{i dz_k \wedge d \bar z_k}{|z_k|^2\log^2 |z_k| ^2} + \frac{1}{B}\alpha\right)\\
& \ge & 2s |K|A^{-1} + 2s(AB)^{-1}\tr_{\om_{\rm lc}} \alpha 
\end{eqnarray*}
As $\om_{\rm lc}$ dominates some smooth form on $X$, the quantity $\tr_{\om_{\rm lc}} \alpha $ is bounded on $X_0$ so that $2s(AB)^{-1}\tr_{\om_{\rm lc}} \alpha $ can be made as small as we want by scaling the metrics on the divisors as explained above. Therefore one has 
\begin{equation}
\label{eq:tr}
\gamma =\tr_{\omi} (\Theta_{h_L}(L)) \le \frac{1}{2}
\end{equation}
on $X_0$.\\\\

\noindent 
\textbf{Step 3: Controlling the error term}\\
Let us get now to the last step in showing that the term
$$\int_{X_0}|\overline\partial u_\varepsilon |^2_{h}dV_{\omega_\infty}$$
tends to zero as $\varepsilon\to 0$. Since $u$ is holomorphic, we have
$$\bar \partial u_\varepsilon= u\otimes \bar \partial \theta_\varepsilon;$$ 
we recall now that $u \in \mathscr H ^{r,s}_B(X|\D)$,  so we have

\begin{equation}
\label{29}
|\bar \partial u_\varepsilon|^2_h\le C|\bar \partial \theta_\varepsilon|^2_{\omi}.
\end{equation}
By inequality \eqref{27} above we infer 
\begin{equation}
\label{30}
\int_{X_0}|\overline\partial u_\varepsilon |^2_hdV_{\omega_\infty}
\le C\int_{X_0}{|\chi^\prime_\varepsilon (\rho)|^2\over \log^2 {1\over \prod_i |s_i|^2}}\left(\sum_j{1\over |s_j|^{2(1-a_j)}}+\sum_k \log^2 |s_k|^2 \right)dV_{\omega_\infty}.
\end{equation}

\noindent As $\omega_\infty$ as mixed Poincaré and cone singularities along $\D$, we have:
\begin{equation}
\label{31}
\int_{X_0}|\overline\partial u_\varepsilon |^2_hdV_{\omega_\infty}
\le C\int_{X_0}\frac{|\chi^\prime_\varepsilon (\rho)|^2  \left(\sum_j{1\over |s_j|^{2(1-a_j)}}+\sum_k \log^2 |s_k|^2 \right)}{\prod_j |s_j|^{2a_j} \prod_k |s_k|^2 \log^2 |s_k|^2 \cdot
\log^2 {1\over \prod_i |s_i|^2}}dV_{\omega}.
\end{equation}
for some constant $C> 0$ independent of $\varepsilon$; here we denote by 
$\omega$ a smooth hermitian metric on $X$. 
We remark that the support of the function $\displaystyle \chi^\prime_\varepsilon (\rho)$ is contained in the set 
$$e^{-e^{1+1/\varepsilon}}\le \prod_i|s_i|^2\le e^{-e^{1/\varepsilon}}$$
so in particular we have
\begin{equation}
\label{32}
{|\chi^\prime_\varepsilon (\rho)|^2\over 
\log^{1\over 2} {1\over \prod_j |s_i|^2}}\le Ce^{-{1\over 2\varepsilon}}.
\end{equation}
We also notice that for each indexes $j_0\in J$ and $k_0\in K$ we have respectively:

\[\int_{X_0}{dV_{\omega}\over |s_{j_0}|^{2}\log^{3/2} \left({1\over \prod_i |s_i|^2}\right)\prod_{j\neq j_0} 
|s_j|^{2a_j}\prod_k |s_k|^2 \log^2 |s_k|^2}\]
\[ \le C \int_{X_0}{dV_{\omega}\over |s_{j_0}|^{2}\log^{3/2} \left({1\over |s_{j_0}|^2}\right)\prod_{j\neq j_0} 
|s_j|^{2a_j} \prod_k |s_k|^2 \log^2 |s_k|^2}\]
and 
\[\int_{X_0}{dV_{\omega}\over |s_{k_0}|^{2}\log^{3/2} \left({1\over \prod_i |s_i|^2}\right)\prod_j
|s_j|^{2a_j}\prod_{k\neq k_0} |s_k|^2 \log^2 |s_k|^2}\]
\[ \le C \int_{X_0}{dV_{\omega}\over |s_{k_0}|^{2}\log^{3/2} \left({1\over |s_{k_0}|^2}\right)\prod_{j} 
|s_j|^{2a_j} \prod_{k\neq k_0} |s_k|^2 \log^2 |s_k|^2}\]
and the integral in the right hand sides are convergent, given that the hypersurfaces 
$(\D_i)$ have strictly normal intersections. \\ 
\noindent
Finally we combine the inequalities \eqref{31}-\eqref{32}, and we get
\begin{equation}
\label{33}
\int_{X_0}|\overline\partial u_\varepsilon |^2dV_{\omega_\infty}
\le Ce^{-{1\over 2\varepsilon}}.
\end{equation}

\noindent 
\textbf{Step 4: Conclusion}\\
As we can see, the relations \eqref{28} and \eqref{33} combined with the fact, coming from \eqref{eq:tr}, that
$$\langle \mathcal R(u_\varepsilon), u_\varepsilon\rangle_h+\gamma |u_{\ep}|^2_h\le  \left(\frac 1 2 +s- r\right)|u_{\ep}|^2_h$$
(which tends to $(\frac 1 2 +s- r)|u|^2_h$) will give a contradiction if $u$ is not identically zero on $X_0$ (we recall that by hypothesis 
we have $r\ge s+ 1$).

\end{proof}

\bibliographystyle{smfalpha}
\bibliography{Biblio.bib}
\vspace{3mm}

\end{document}